\documentclass[11 pt]{amsart}

\usepackage{amsmath}
\usepackage{amsxtra}
\usepackage{amscd}
\usepackage{amsthm}
\usepackage{amsfonts}
\usepackage{amssymb}
\usepackage{eucal}

\newtheorem{thm}{Theorem}[section]

\newtheorem{cor}[thm]{Corollary}
\newtheorem{lem}[thm]{Lemma}
\newtheorem{lemma}[thm]{Lemma}
\newtheorem{prop}[thm]{Proposition}
\newtheorem{proposition}[thm]{Proposition}

\newcommand\alp{\alpha}     

\newcommand\gam{\gamma}     
\newcommand\del{\delta}     \newcommand\Del{\Delta}

\newcommand\lam{\lambda}        \newcommand\Lam{\Lambda}

\newcommand\calB{{\mathcal{B}}}

\newcommand\calD{{\mathcal{D}}}
\newcommand\calE{{\mathcal{E}}}
\newcommand\calF{{\mathcal{F}}}

\newcommand\calH{{\mathcal{H}}}
\newcommand\calI{{\mathcal{I}}}

\newcommand\calM{{\mathcal{M}}}
\newcommand\calN{{\mathcal{N}}}
\newcommand\calO{{\mathcal{O}}}
\newcommand\calP{{\mathcal{P}}}

\newcommand\calS{{\mathcal{S}}}
\newcommand\calT{{\mathcal{T}}}

\newcommand\calV{{\mathcal{V}}}

\newcommand\calX{{\mathcal{X}}}

        \newcommand\bfG{{\mathbf G}}
        \newcommand\bfH{{\mathbf H}}

        \newcommand\bfM{{\mathbf M}}

\newcommand\QQ{\mathbb{Q}}

\newcommand\TT{\mathbb{T}}

\newcommand\PP{\mathbb{P}}
\renewcommand\AA{\mathbb{A}}

\newcommand\GG{\mathbb{G}}

\newcommand\ZZ{\mathbb{Z}}

\newcommand\CC{\mathbb{C}}

\newcommand\BB{\mathbb{B}}

 \newcommand\grb{{\mathfrak{b}}}

 \newcommand\grg{{\mathfrak{g}}}

 \newcommand\grt{{\mathfrak{t}}}

\newcommand\sdp{\times \hskip -0.3em {\raise 0.3ex
\hbox{$\scriptscriptstyle |$}}} 

\newcommand\Aut{\operatorname{Aut}}

\newcommand\End{\operatorname{End\,}}

\newcommand\Hom{\operatorname {Hom}}

\newcommand\Sym{\operatorname{Sym}}

\newcommand\oM{{\overline{M}}}

\newcommand\os{{\overline{s}}}

\newcommand\ow{{\overline{w}}}
\newcommand\oW{{\overline{W}}}

\newcommand\hatB{{\widehat{B}}}

\newcommand\tilM{{\widetilde{M}}}

\newcommand\tils{{\widetilde{s}}}

\newcommand\tilw{{\widetilde{w}}}

\newcommand\Tdbar{\overline{T}^{\vee}}

\newcommand\x{\times}

\newcommand{\Mbar}{\overline{M}}
\newcommand\tcS{\widetilde{\calS}}
\renewcommand\Sym{\operatorname{Sym}}
\newcommand\St{\mathbf{St}}
\newcommand\Coh{\operatorname{Coh}}
\newcommand\reg{\operatorname{reg}}
\renewcommand\lg{\operatorname{lg}}
\newcommand\sh{\operatorname{sh}}
\renewcommand\AA{\mathbb A}
\newcommand\ocM{\overline{\calM}}
\newcommand\CM{{\mathcal C}{\mathcal M}}
\newcommand\ev{\operatorname{ev}}
\newcommand\onabla{\overline{\nabla}}
\newcommand\tcH{\widetilde{\calH}}
\renewcommand\BB{\calB}

\newcommand\bT{\mathbf{T}}
\newcommand\bS{\mathbf{S}}

\newcommand\cL{\mathcal{L}}

\newcommand\Td{T^{\vee}} 

\newcommand\Ba{B_{\widehat{W}}}  
\newcommand\Pic{\operatorname{Pic}}
\newcommand\Hilb{\operatorname{Hilb}}
\newcommand\Lie{\operatorname{Lie}}
\newcommand\Kb{\overline{\mathcal{K}}}

\newcommand\bG{\mathbf{G}}
\newcommand\dR{H^2_\textup{deRham}}

\newcommand\egt{\mathbf{t}}

\usepackage[all]{xy}

\usepackage{amsmath}

\usepackage{amsxtra}

\usepackage{amscd}

\usepackage{amsthm}

\usepackage{amsfonts}

\usepackage{amssymb}

\usepackage{eucal}
\usepackage{dsfont}

\newcommand\IC{\operatorname{IC}}
\newcommand\Spr{\operatorname{Spr}}
\newcommand\tX{\widetilde{X}}

\begin{document}

\title{Quantum cohomology of the Springer resolution}
\author{Alexander Braverman, Davesh Maulik and Andrei Okounkov}
\begin{abstract}
Let $G$ denote a complex, semisimple, simply-connected group and
$\calB$ its associated flag variety.  We identify the equivariant
quantum differential equation for the cotangent bundle
$T^*\calB$ with the affine Knizhnik-Zamolodchikov connection of
Cherednik and Matsuo. This recovers Kim's description of quantum
cohomology of $\calB$ as a limiting case. A parallel result is
proven for resolutions of the Slodowy slices. Extension to arbitrary
symplectic resolutions is discussed.
\end{abstract}
\maketitle
\section{Introduction}


The main goal of this paper is to study the quantum cohomology of the Springer resolution associated to a semisimple algebraic group $G$.  Since many of the ideas and constructions here apply to arbitrary symplectic resolutions, much of the introduction will discuss this more general context.

\subsection{Geometry of symplectic resolutions}\label{geom_sympl}
Recall from \cite{Kaledin2} that
a smooth algebraic variety $X$ with a holomorphic symplectic form $\omega\in H^0(X,\Omega^2)$
is called a \emph{symplectic resolution} if the canonical map
\begin{equation}
X \to X_0=\textup{Spec} H^0(X,\mathcal{O}_X)\label{defpi}
\end{equation}
is projective and birational.

The deformations of the pair
$(X,\omega)$ are classified by the image $[\omega]$
of the symplectic form in $\dR(X)$. In many important examples,
the universal families

\begin{equation}\label{deform}
\xymatrix{
X \ar@{^{(}->}[r]\ar[d]& \tX \ar[d]_\phi \\
[\omega] \ar@{^{(}->}[r] & H^2(X,\CC) \\
}
\end{equation}
may be described explicitly, see \cite{Kaledin2} for further discussion.
The generic fiber of \eqref{deform} is affine. Fibers
with algebraic cycles $\alpha^\vee \in H_2(X,\ZZ)$
lie over hyperplanes
\begin{equation}
 (\omega,\alpha^\vee)= \int_{\alpha^\vee} \omega = 0
\label{hyperpl}
\end{equation}
in the base. Primitive effective classes $\alpha^\vee$ as above
are called \emph{primitive coroots} by the analogy with the following
prominent example.

\subsection{Grothendieck resolution}\label{sGroth}
Let $G$ be a complex, semi-simple  simply connected
algebraic group with Lie algebra $\grg$. Let $\calB$ be the flag
variety of $G$ that parametrizes, among other things, Borel
subalgebras $\grb\subset\grg$. Consider
$$
\tX = \{(\grb,\xi),\, \xi \in [\grb,\grb]^\perp\} \subset \calB\times\grg^*  \,,
$$
and the map
\begin{equation}
\tX \owns (\grb,\xi) \overset\phi\longrightarrow
\xi \big|_{\grb/[\grb,\grb]} \in \grt^*\,, \label{GrothPhi}
\end{equation}
where $\grt$ is the Lie algebra of a maximal torus $T\subset G$.

The generic fiber of $\phi$ is a coadjoint orbit of $G$ with its
canonical Kirillov-Kostant form $\omega$, while
$$
X= \phi^{-1}(0) \cong T^*\calB
$$
with the canonical exact symplectic form of a cotangent bundle.
The map
\begin{equation}
  \tX \owns (\grb,\xi)\longrightarrow \xi \in \grg^*
\label{Groth}
\end{equation}
takes $X$ to the nilpotent cone $\calN\subset\grg^*$. This is known
as the \emph{Springer resolution} of $\calN$, while \eqref{Groth}
is referred to as the Grothendieck simultaneous resolution.

Fix a Borel subgroup $B\subset G$ with a maximal torus $T$.
A weight $\lambda\in \Hom(B,\CC^*)$ defines an equivariant
line bundle
$$
\calO(\lambda)=G \times_{B} \CC_{\lambda}
$$
on $\calB$ and, by pullback, on $X$. The map
$$
\lambda \mapsto D_\lambda = c_1(\calO(\lambda))
$$
yields an isomorphism
\begin{equation}
  \Lambda = \Hom(T,\CC^*) \cong H^2(X,\ZZ)
\label{Lam}
\end{equation}
Applying $\otimes_\ZZ\CC$ to  \eqref{Lam}, we get
an isomorphism $\grt^* \cong H^2(X,\CC)$ that
makes
\begin{equation}\label{GrothPhi2}
\xymatrix{
T^*\calB \ar@{^{(}->}[r]\ar[d]& \tX \ar[d]_\phi \\
0 \ar@{^{(}->}[r] & \grt^* \\
}
\end{equation}
an instance of the diagram \eqref{deform} with
$\phi$ as in \eqref{GrothPhi}.

We will denote by $R, R_+,$ and $\Pi$ the set of roots, positive roots, and simple roots respectively.
A positive coroot $\alpha^\vee$
defines an $SL_2$-subgroup $G_{\alpha^\vee} \subset G$
and hence a rational curve
$$
G_{\alpha^\vee} \cdot [B] \subset \calB \subset X\,.
$$
Since $G$ is simply-connected, coroots generate the lattice
\begin{equation}
  \Lambda^\vee = \Hom(\CC^*,T) \cong H_2(X,\ZZ) \,,
\label{Lamv}
\end{equation}
in which the positive ones generate the effective cone.
It is well-known that this classical notion of a positive
coroot agrees with the definition given in Section \ref{geom_sympl}.

\subsection{Equivariant symplectic resolutions}

Our interest lies in enumerative invariants of curves in $X$,
particularly, in quantum cohomology of $X$. Since a symplectic resolution $X$ may be deformed to an affine variety,
most Gromov-Witten invariants of $X$ trivially vanish.
However, \emph{equivariant} Gromov-Witten invariants of $X$ may
be very interesting.

In what follows,
we assume that a connected reductive algebraic group $\bG$ acts on $X$ so
that:
\begin{enumerate}
\item $X^g$ is proper for some $g\in\bG$;
\item  $\bG$ scales $\omega$
by a \emph{nontrivial character}.
\end{enumerate}
The second property implies $X$ does not have any
$\bG$-equivariant symplectic deformations since any connected group
acts trivially on $H^2(X)$. See \cite{Kaledin3}
for general conjectures concerning the existence of
scaling actions on $X$.

For the
Springer resolution $X=T^*\calB$, we take
$$
\bG = G \times \CC^*
$$
where the second factors acts by scalar operators on $\grg^*$, $\grt^*$,
and the cotangent fibers. While all constructions of Section \ref{sGroth} are
$\bG$-equivariant, note that $\bG$ acts nontrivially on the base
of the deformation with a unique fixed point $0\in\grt^*$.

Other examples of equivariant symplectic resolution include:
cotangent bundles of homogeneous spaces, as well as Hilbert schemes of
points and more general moduli of sheaves on symplectic
surfaces. The largest class of examples comes from
Nakajima quiver varieties, see
\cite{ginzl} for an introduction. For a general symplectic
resolution, we denote by $G\subset\bG$ the stabilizer of $\omega$.

\subsection{Goal of the project}

There exist powerful heuristic as well as proven principles
that organize Gromov-Witten invariants of algebraic varieties
into certain beautiful and powerful structures. This is loosely
known as \emph{mirror symmetry},  see e.g.\ \cite{CoxKatz,Gross,mirror} for an
introduction. We believe for symplectic resolutions, these
structures recover and generalize some classical
notions of geometric representation theory.

To see this in full generality is the goal of
a large joint project which we pursue with R.~Bezrukavnikov, P.~Etingof,
V.~Toledano-Laredo, and
others. It has also many points of contact with the ongoing work of
N.~Nekrasov and S.~Shatashvili \cite{NekSh}.

This paper may be seen as a first step, in which
we explain what these structures mean
 for the most classical symplectic
resolution of all --- the  Springer resolution.
This will be reflected in
the paper's structure: the reader will notice that the generalities take up much more
space than the actual geometric computation of quantum multiplication
in Section \ref{proofoftheorem}.

\subsection{Quantum multiplication}\label{quantum-int}

By definition, the operator of quantum multiplication by
$\alpha\in H^*_\bG(X)$ has the following matrix elements
\begin{equation}
  (\alpha\circ\gamma_{1},\gamma_2) =
\sum_{\beta \in H_{2}(X,\ZZ)}
q^{\beta}\,
\langle \alpha, \gamma_{1}, \gamma_{2}\rangle^{X}_{0,3,\beta}
\,,\label{qmult}
\end{equation}
where $(\cdot\,,\,\cdot)$ denotes the standard inner product on cohomology and the
quantity in angle brackets is a $3$-point, genus $0$, degree $\beta$ equivariant
Gromov-Witten invariant of $X$, see Section \ref{GWprelim}.
It is a very general fact
that \eqref{qmult} defines a commutative associative deformation of the algebra $H^*_{\bG}(X)$, see e.g.\ \cite{CoxKatz}.

We denote by $T^\vee$ the dual torus, for which the roles of
\eqref{Lam} and \eqref{Lamv} are reversed. Each monomial $q^\beta$
is naturally a function on $T^\vee$. In fact, the series \eqref{qmult}
will be shown to be rational, that is, to lie in $\CC(T^\vee)$.
Parallel rationality is expected for arbitrary symplectic resolutions.

It is well-known that $H^*(\calB)$ is generated by divisors and, for
an arbitrary symplectic resolution $X$, one may hope that the
quantum cohomology is generated by divisors even if the classical
cohomology is not.

The main result of the paper is a computation of operators of
quantum multiplication by divisors for the Springer resolution.
Before stating it, we discuss the general structure of the answer.

\subsection{Correspondence algebra and root subalgebras}\label{rootsubalgebra}

Consider the fiber product
\begin{equation}
X\times_{X_0} X \subset X \times X\label{XX} \,.
\end{equation}
It is known that all components of \eqref{XX}
have dimension
$\le \dim X$ and those of dimension $\dim X$ are
Lagrangian, see \cite{ginzl}.
Lagrangian components of \eqref{XX} act by correspondences on the
Borel-Moore homology of $X$.  We
denote by
$\mathrm{Corr}(X) \subset \End(H^*(X))$ the subalgebra that they generate, known as the \emph{correspondence algebra}.  We will show that the quantum contribution to divisor multiplication will lie in $\mathrm{Corr}(X)$.  More precisely, we establish the following refined statement in section \ref{proofoftheorem}.

Let $X^\alpha$ be a generic fiber of \eqref{deform} over the
hyperplane where $\alpha^\vee$ is algebraic. Then we can again consider the fiber product
\begin{equation}
  X^\alpha \times_{X^\alpha_0} X^\alpha \subset
X^\alpha \times X^\alpha \label{XaXa} \,.
\end{equation}
Since $X^\alpha \sim X$ as a topological $G$-space,
Lagrangian components of \eqref{XaXa} act by
correspondences in the Borel-Moore homology of $X$.
We denote $\bS^\alpha \subset \mathrm{Corr}(X)$ the
subalgebra that they generate.

Let $D\subset H^2(X)$ be a divisor. By an abstract
argument, we show the operator  $D\circ$ of quantum
multiplication by $D$ has the following form
\begin{equation}
  \label{abstrD}
  D\circ  = \textup{classical} + t \sum_{\alpha^\vee} \,
(D,\alpha^\vee) \,
S_\alpha\!\left(q^{\alpha^\vee}\right)\,,
\quad S^\alpha(z) \in \bS^\alpha\otimes
\CC[[z]] \,,
\end{equation}
where the first term is the classical multiplication by $D$,
the sum is over all positive coroots $\alpha^\vee\in H_2(X,\ZZ)$,
and $t$ is the weight of the symplectic form $\omega$.  See sections \ref{Lcorr} and \ref{simultaneous} for the argument.

Formula \eqref{abstrD} is \emph{functorial} in the sense
that $\alpha^\vee$-term in it describes the
quantum corrections to multiplication in $H^*(X^\alpha)$.
This reduces the computation of quantum cohomology
to varieties with a unique effective curve class, as we discuss further below.
We will also see that similar functoriality holds for slices
of $X$.

In the case of Nakajima varieties, which come in families, the root subalgebras $\bS^\alpha$
and the operators of classical multiplication can be extended to a representation
of the so called \emph{Yangian} algebra $Y(\grg)$ associated to the corresponding
Kac-Moody Lie algebra, see \cite{var}.
Through the work of Nekrasov and Shatashvili, this has
a direct link to the classical work of C.~N.~Yang in quantum
integrable systems.

\subsection{Main result}

In the Springer case, the action of classical multiplication and the root subalgebras
has a very concrete description in terms of
 the \emph{graded affine Hecke algebra} $\calH_t$.
Recall that $\calH_t$ is generated by the symmetric algebra $\Sym \grt^*$
of $\grt^*$
and the group algebra of $W$ subject to the relation
$$
s_{\alp} x_\lam-x_{s_{\alp}(\lam)}s_{\alp}=t(\alp^{\vee}, \lam),
$$
where $s_{\alp}$ is the reflection associated with the simple root $\alp$
and $x_\lambda \in S^1 \grt^*$ corresponds to $\lambda\in \grt^*$.

As before, $t$ is a parameter which may be identified with the weight of the symplectic form $\omega$,
that is, with the equivariant parameter of the $\CC^*$-factor in
$\bG$.
The algebra $\calH_t$ is graded by  $\deg x_\lam=2,\deg w=0, \deg t=2$.

An action of $\calH_t$ on $H^*_\bG(T^*\calB)$
is due to Lusztig \cite{lus-cusp}.
Its construction is recalled in Section \ref{hecke-geometry}.
In particular, $x_{\lam}$ acts by classical multiplication by $D_{\lam}$,
while $\bS^\alpha$ is generated by the corresponding reflection in
$s_\alpha\in W$.
Having introduced the above notations we can now state

\begin{thm}\label{formula} The operator of quantum multiplication
by the divisor $D_\lambda$ is given by
\begin{equation}
    D_{\lambda}\circ \,  =x_{\lam} + t \sum\limits_{\alp^{\vee}\in R_+^{\vee}}
(\lambda, \alpha^{\vee}) \frac{q^{\alpha^{\vee}}}{1 - q^{\alpha^{\vee}}} (s_{\alpha}-1) .
\label{mainf}
\end{equation}
\end{thm}

Since $D_\lambda$ generate cohomology, this determines all operators
of quantum multiplication.  We give a more detailed statement of this theorem in section \ref{main-revisited}.
The commuting elements \eqref{mainf} are, of course, well-known, as we now recall.
\subsection{The quantum connection}
Over
$$
T^{\vee}_{\reg}=\{ q\in T^{\vee}|\ \alp(q)\neq 1\text{ for any }\alp\in R_+\}.
$$
consider what is known as the \emph{quantum differential equation}, namely the following
connection on the trivial bundle with fiber $H^*_{\bG}(X)$
\begin{equation}\label{formula-conn}
\nabla_{\lam}=d_{\lam}-D_{\lam}\circ=d_{\lam}-x_{\lam}- t \sum_{\alp\in R_+}
(\lambda, \alpha^{\vee}) \frac{q^{\alpha^{\vee}}}{1 - q^{\alpha^{\vee}}} (s_{\alpha}-1).
\end{equation}
Here $d_{\lam}$ is the derivative in the direction of $\lambda\in \grt^* = \textup{Lie}(T^\vee)$.
It is a general result of Dubrovin that quantum connections are always flat, see \cite{CoxKatz}.

In our case, a simple gauge transformation, see Section  \ref{affine KZ}, identifies \eqref{formula-conn}
with the {\em affine KZ connection} studied by Cherednik \cite{Cher-book}, Matsuo \cite{Mats},
Felder and Veselov \cite{FelderVeselov} and others.

Recall that an integrable connection $\nabla$ on an algebraic vector bundle $\calE$ defines a $D$-module,
that is, a module over differential operators on the base. For the quantum differential
equation, this is known as the \emph{quantum $D$-module}. The isomorphism class of the $D$-module
determines $\nabla$ up to gauge equivalence; in other words, passing from the quantum connection to the quantum
$D$-module corresponds to forgetting the fact that the quantum connection was defined on the trivial
bundle (thus, knowing the quantum $D$-module alone is not enough in order to recover the quantum
multiplication by divisor classes).

In the Springer case, the
quantum $D$-module coincides with the quantum Calogero-Moser $D$-module for the Langlands
dual group.  After providing background, we will give a precise formulation in Section \ref{calogero} and Theorem \ref{main-detailed}.
In particular, one can describe the corresponding quantum cohomology ring of $T^*\calB$ by using
the corresponding {\em classical Calogero-Moser integrable system}.
For $\grg=\mathfrak{sl}_n$, this was obtained
independently by Nekrasov and Shatashvili. A related result is due to A.~Negut \cite{Negut}, see below.

\subsection{The K\"ahler moduli space}

Let $\Kb=\Tdbar \supset \Td$ be the
compactification defined by the fan of Weyl chambers.
The connection \eqref{formula-conn} extends to $\Kb$ as a connection
with first-order poles along a normal crossing divisor.
In accordance with the mirror symmetry nomenclature, we call
$\Kb$ the \emph{compactified K\"ahler moduli space}.

The expansion \eqref{qmult} is around the point
$$
q^{\,\alpha_1^\vee} = q^{\,\alpha_2^\vee} = \dots = 0\,,
$$
where $\alpha_i^\vee$ are the simple positive coroots.
This is one of the $|W|$-many  $\Td$-fixed points of $\Kb$.
Such points are called the \emph{large radius points}.

The Weyl group $W$ acts on both
the base and the fiber of $\nabla$ and one easily checks that
\begin{equation}
  \label{Wequiv}
  w \, \nabla_\lam \, w^{-1} = \nabla_{w(\lam)} \,, \quad w\in W \,.
\end{equation}
In fact, this is \emph{equivalent} to the commutation relations in $\calH_t$.

Note that $T^\vee$ is naturally the base of the multiplicative version of Grothendieck simultaneous
resolution for the Langlands dual group $G^\vee$. In fact, we expect
these two families to be \emph{equivariant mirrors} and hope to elaborate on this point in
a future paper.

\subsection{Monodromy and derived equivalences}\label{smdr}

In general, a compactified K\"ahler moduli space may have large
radius points $m_1$ and $m_2$ that correspond to nonisomorphic varieties
$X_1\not\cong X_2$. One always expects, however, that
$$
D^b(X_1) \cong D^b(X_2) \,.
$$
and, moreover, that the equivalence
should depend on a choice of a path from $m_1$ to $m_2$, thus
giving
\begin{equation}
\rho: \pi_1(U) \longrightarrow \Aut D^b(X_i)\label{pitoD}
\end{equation}
for a certain open $U\subset\Kb$. One further expects that $U$ is the
regular locus of $\nabla$ and that there is a nonzero
intertwiner between the monodromy of $\nabla$ and the
image of $\rho$ in $K$-theory
$$
\rho_K: \pi_1(U) \longrightarrow \Aut K(X)\otimes \CC \,.
$$
See, for example, \cite{Horja} for an introduction to these ideas.

Specializing this to the
Springer case, we first use the $W$-equivariance
\eqref{Wequiv} to descend $\nabla$ to a connection on
$T^{\vee}_{\reg}/W$. The fundamental group of $T^{\vee}_{\reg}/W$ is
the affine braid group $\Ba$.
Bezrukavnikov constructed an action of $\Ba$ on $D^b(X)$ in
\cite{Bezr,Bezr-Riche}, see also the work of
Khovanov and Thomas \cite{KhovanovThomas}
in the case $\grg=\mathfrak{sl}_n$.
On the $K$-theory, this braid group action
factors trough the affine Hecke
algebra action
constructed earlier by  Kazhdan and Lusztig
\cite{kazhdan-lusztig} and Ginzburg \cite{chriss-ginz}. And indeed this
precisely
matches Cherednik's description of the monodromy, see page 64 in
\cite{Cher-book} and Sections
\ref{monodromy} and \ref{monodromy-derived} for further discussion.

A parallel result for $X=\Hilb(\CC^2,n)$ is the subject of \cite{Bezr-O}.
In both cases, monodromy of $\nabla$ is irreducible and
isomorphic to $\rho_K$.

\subsection{Matching of equivariant parameters}
One aspect of the correspondence described in Section \ref{smdr} merits
a special discussion, namely the identification of equivariant
parameters.

Equivariant $K$-theory $K_\bG(X)\otimes\CC$ is a module over
the representation ring $\CC[\bG]^\bG$ of $\bG$
while $H^*_\bG(X,\CC)$ is a module over $\CC[\mathbf{g}]^\bG$ where
$\mathbf{g}$ is the Lie algebra of $\bG$. Assuming that $\nabla$ depends
regularly on $a\in\mathbf{g}$, its monodromy is an entire function
of $a$. Hence, to match $\rho_K$ with the monodromy of $\nabla$,
we need a map
$$
\zeta: \CC[\bG]^\bG \to \CC_\textup{an}[\mathbf{g}]^\bG\,,
$$
where $\CC_\textup{an}$ stands for analytic functions.
We claim the only natural choice is
\begin{equation}
  \label{zeta}
   \zeta(f)(a) = f\left(e^{2\pi i a}\right) \,.
\end{equation}
This holds in the Springer case and can be argued in general as
follows.

The simplest autoequivalences of $D^b(X)$ are tensor products
with line bundles $\cL\in\Pic(X)$. These are expected to be
compatible with $\rho$ as follows. Let $m\in\Kb$ be a large
radius point. The intersection of a small
neighborhood of $m$ with $U$ has an abelian fundamental group,
namely $H^2(X,\ZZ)$. One requires that
\begin{equation}
\rho(c_1(\cL)) = \cL \otimes \, \textup{---}
\label{cLm}
\end{equation}
To see the connection with \eqref{zeta}, we treat the
purely equivariant classes in $H^2_\bG(X,\ZZ)$ on the same
footing as the geometric ones. For a purely equivariant
class, the quantum differential equation is
a constant coefficient, in fact, scalar equation.
In this case \eqref{cLm}
specializes to \eqref{zeta}.

\subsection{Shift operators}\label{shift-intro}
Formula \eqref{zeta} predicts the monodromy is a
periodic function of equivariant parameters. For
the Springer resolution, this means
that for any $s\in \Lambda^\vee\oplus \ZZ$, we have a
a shift operator
$$
\bS(s) \in \End(H^*_T(X)) \otimes \CC(T^\vee)
$$
satisfying
$$
\nabla(a) \, \bS(s) = \bS(s) \, \nabla(a+s) \,,
$$
where $a\in\Lie(\bG)$ denotes the
equivariant parameters of $\nabla$.
And, indeed, such intertwiner may be constructed geometrically
as described in Section \ref{shift}.

These are the shift operators of Opdam (\cite{opdam, heckman}) that played a very
significant role in understanding the quantum Calogero-Moser
systems.

\subsection{The Toda limit}\label{int-limit}
Recall that the parameter $t$ in \eqref{mainf} is the equivariant parameter
for the scaling action in the fibers of $T^*\calB$. It is known and explained
in Section \ref{limit-geometry} that the $t\to\infty$ limit, taken
correctly, recovers
the quantum cohomology of the base $\calB$.

The analog of $\calH_t$ for $\calB$ is a certain nil-version $\calH_{nil}$
generated by  $x_{\lam}$ for $\lam\in\grt^*$ and $\ow$ for $w\in W$.
The element  $x_{\lam}$ still acts on by multiplication
by a divisor $D_{\lam}^{\calB}$, while $\ow$ are nilpotent in $\calH_{nil}$.
Precise definitions are given in Section \ref{hnil}.

Referring to Sections \ref{limit-algebra} and \ref{limit-geometry} for
the details of the $t\to\infty$ limit, we state the result.
Introduce the following
subset $R_+'$ of the set of $R_+$ of positive roots of $G$. Namely, let $\alp\in R_+$. Denoting by $\Pi$ the set of simple positive roots, we
have an expansion
$$
\alp=\sum\limits_{\beta\in \Pi} a_{\beta}\beta \,.
$$
Then $\alp\in R_+'$ if one of the following conditions is satisfied:
\begin{enumerate}
\item $\alp$ is a long root;
\item $a_{\beta}=0$ for all long simple roots $\beta$.
\end{enumerate}

In particular, we see that $R_+=R_+'$ if $G$ is simply laced and that $R_+'$ contains all simple roots.

\begin{thm} \label{formula-flags}
The operator of quantum multiplication by $D_{\lam}^{\calB}$ on $H^*_G(\calB)$ is equal to
$$
x_{\lam}+\sum\limits_{\alp\in R_+'}(\lam,\alp^{\vee})q^{\alp^{\vee}}\os_{\alp}.
$$
\end{thm}

\noindent

The quantum cohomology of $\calB$ has been studied in great detail by many
authors, see for example \cite{Ciocan,GiventalKim,kim,mih,Rietsch} for a very incomplete selection of
references.  In particular, the above formula may be deduced from Theorem 6.4 in \cite{mih}.

\subsection{Generalization to Slodowy slices}\label{int-slodowy}

To any $n$ in the nilpotent cone $\calN$,
 one associates the {\em Slodowy slice} $\calS_n\subset\calN$, which is a transversal slice
to the $G$-orbit of $n$ in $\calN$. This is an affine conical algebraic variety, which has an open smooth symplectic
subset (equal to the intersection of $\calS_n$ with the open $G$-orbit in $\calN$).
The slice $\calS_n$ is defined uniquely up to conjugation by an element of the centralizer $Z_n$ of $n$ in $G$.
Let $\tcS_n=\pi^{-1}(\calS_n)$. Then $\tcS_n$ is smooth and symplectic; it provides a symplectic resolution of singularities
for $\calS_n$.  Note that when $n=0$ we have $\calS_n=\calN$ and $\tcS_n=T^*\calB$.

According to \cite{lus-cusp} and \cite{ginzburg} the algebra
$\calH_t$ still acts on $H^*_{Z_n\x\CC^*}(\tcS_n)$. Abusing notation we shall denote the restriction of
the class $D_{\lam}$ to $\tcS_n$ by the same symbol. Then we have
\begin{thm}\label{formula-slodowy}
Assume that the restriction map $H^2(X,\ZZ)\to H^2(\tcS_n,\ZZ)$ is an isomorphism.
Then the operator of quantum multiplication by $D_{\lam}$ in $H^*_{Z_n\x\CC^*}(\tcS_n)$ is given by the same
formula as in Theorem \ref{formula}.
\end{thm}

Let us note that the assumption on $H^2$ in the formulation of Theorem \ref{formula-slodowy} is not very restrictive.
For example, one can show (cf. the Appendix) that this assumption always holds when $G$ is simply laced.

\subsection{Reduction to rank 1}\label{rank1}

As discussed in section \ref{rootsubalgebra}, one can study quantum cohomology of $X$ in terms of
that of generic non-affine deformations $X^{\alpha}$.  In fact, we can further simplify this analysis as follows.

The symplectic form $\omega$ induces a Poisson algebra
structure on $\mathcal{O}_X$ which makes $X_0$ a Poisson
variety with finitely many symplectic leaves, see \cite{Kaledin3}.
Let $Z\subset X_0$ be a symplectic leaf of minimal
dimension.  If one has an isotrivial fibration $X\to Z$ which is a Poisson map,
then, since $Z$ is affine, this reduces quantum cohomology of $X$ to that of the
fiber $X'$.
If one has such a structure for $X^\alpha$, in combination with \eqref{abstrD}, we can reduce further to the fibers
$$
\xymatrix{
X'_\alpha \ar@{^{(}->}[r]& X^\alpha \ar[d] \\
 & Z^\alpha \\
}
$$
for the simpler varieties $X^\alpha$ of $X$.

In general, while we may not have a global fibration $X\to Z$, there is a formal result of Kaledin \cite{Kaledin3} which allows us to write, for $z\in Z$, the formal neighborhood $\widehat{(X_0)}_{z}$
as a product of $\widehat{Z}_{z}$ and the formal neighborhood of a lower-dimensional symplectic singularity.  It turns out that, for equivariant symplectic resolutions, this formal statement is already sufficient for the reduction step.

We say that a symplectic resolution $X$ has rank $1$ if
$Z$ is a point and $\dim H_2(X)=1$. The above reduction
procedure obviously stops once we reach a rank $1$
variety $X$.
Examples of rank 1 varieties are
cotangent bundles $T^*\mathrm{Gr}(k,n)$ of the Grassmannians and
also moduli of
framed torsion-free sheaves on $\CC^2$.
Further examples may be found among Nakajima quiver varieties.

In the Springer case, only $T^*\PP^1$ appear as fibers $X'_\alpha$ - here the formal argument will be unnecessary.
Their quantum cohomology is very well understood. Similarly,
Nakajima varieties for quivers of finite type lead only to
$T^*\mathrm{Gr}(k,n)$. They will be discussed in forthcoming
paper (the corresponding quantum connection should presumably be related to the
{\em trigonometric Casimir connection} -- cf. \cite{TL1, TL2}). Similarly, for quiver varieties
of affine type one has to understand the the quantum cohomology of the moduli space of
framed sheaves on $\PP^2$ of arbitrary rank; this is more involved - cf.  \cite{MO}.

\subsection{Related work}

{}From our viewpoint, the operators of quantum multiplication
form a family (parametrized by $q$) of maximal abelian subalgebras
in a certain geometrically constructed Yangian. Long before
geometric representation theory, Yangians appeared in
mathematical physics as symmetries of integrable models
of quantum mechanics and quantum field theory. The maximal
abelian subalgebra there is the algebra of quantum integrals
of motion, that is, of the operators commuting with the
Hamiltonian. A profound correspondence
between quantum integrable systems and supersymmetric gauge
theories was discovered by Nekrasov and Shatashvili, see \cite{NekSh}.
It connects the two appearances of Yangians and, for example,
correctly predicts that the eigenvalues of the operators of quantum
multiplication are given by solutions of certain Bethe
equations (well-known in quantum integrable systems, but
probably quite mysterious to geometers).
The integrable and geometric viewpoints are complementary in
many ways and, no doubt, will lead to new insights on both
sides of the correspondence.

It is a well-known phenomenon that a particular solution
of the quantum differential equation (the so-called $J$-function)
may often be computed as the generating function of integrals of
certain cohomology classes over different compactifications of the
moduli space of maps $\PP^1\to X$.
For example for maps to flag varies of $\grg=\mathfrak{sl}_n$, one can try to use
the so called Laumon moduli space flags of sheaves on $\PP^1$.
The corresponding generating function
was identified by A.~Negut with
the eigenfunctions of the quantum Calogero-Moser system \cite{Negut}.

As a rule, sheaf compactifications admit torus action with isolated
fixed points, thus expressing $J$-functions
as a certain multivariate hypergeometric series. The spectrum
of the operators of quantum multiplication may be deduced from
the semiclassical asymptotics of this solution. This asymptotic behavior
is determined by a unique maximal term in the hypergeometric
series. The spectrum is thus determined as a solution of a
certain finite-dimensional variational problem. It is one of
the cornerstones of Nekrasov-Shatashvili theory that this
variational problem coincides with the Yang-Yang variational
description of Bethe roots. For flag varieties, this can
be seen very explicitly.

\subsection{Organization of the paper}
In Section \ref{hecke} we review necessary facts about the graded affine
Hecke algebra and the corresponding affine KZ connection as well as introduce the  graded affine nil-Hecke
algebra and the corresponding connection. In Section \ref{hecke-geometry} we
review various well-known
facts about the cohomology of $\calB$ and its cotangent bundle and restate our main result in these terms. In Section \ref{GWprelim}
we discuss equivariant Gromov-Witten invariants and explain general properties
of the reduced virtual fundamental class.  Once these generalities are in place, we apply them in Section \ref{proofoftheorem} to give an extremely short proof of Theorem \ref{formula}.
 In Section \ref{shift} we explain the geometric construction of the shift operators for the affine KZ connection.
In Sections \ref{limit-algebra} and \ref{limit-geometry} we study the $t\to \infty$ limit of the above constructions.

Finally, in Section \ref{slodowy} we prove the generalization of
Theorem \ref{formula} to the varieties $\tcS_n$.

\subsection{Acknowledgments}
Numerous discussions with R.~Bezrukavnikov, P.~Etingof, N.~Nekrasov,
and S.~Shatashvili
played a very important role in development of the ideas presented
in this paper.
We would also like to thank T.~Bridgeland, I.~Cherednik, D.~Kaledin, T.~Lam,  E.~Opdam, R.~Pandharipande, and Z.~Yun for very helpful conversations and for providing
us with references on various subjects.

A.B. was partially supported by the NSF grant DMS-0901274.
D.M. was partially supported by a Clay Research Fellowship.
A.O. was partially supported by the NSF grant DMS-0853560.

Part of this work was done when A.B. was visiting the program on "Algebraic Lie Theory" at
Isaac Newton Mathematical Institute. He would like to thank the program organizers for the invitation and
the institute stuff for its hospitality.

\section{Hecke algebras, connections and integrable systems}\label{hecke}
In this section we recall some standard algebraic background about graded affine Hecke algebras and the affine KZ connection.  The geometrically-minded reader can skip this section on a first reading.

\subsection{The graded affine Hecke algebra}
Consider the graded Hecke algebra $\calH_t$. By definition, it is generated by elements $x_{\lam}$ for $\lam\in\grt^*$
and $w\in W$ and a central element $t$ such that

a) $x_{\lam}$ depends linearly on $\lam\in\grt^*$

b) $x_{\lam} x_{\mu}=x_\mu x_\lam$ for any $\lam,\mu\in\grt^*$

c) The $w$'s form the Weyl group inside $\calH_t$;

d) For any $i\in I$, $\lam\in \grt^*$ we have
$$
s_i x_\lam-x_{s_i(\lam)}s_i=t(\alp_i^{\vee}, \lam),
$$
where $s_i$ is the reflection associated with the simple root $\alp_i$.

We define a grading on $\calH_t$ in such a way that $\deg x_\lam=2,\deg w=0, \deg t=2$.
\subsection{The affine KZ connection}\label{affine KZ}
The Dunkl (or affine KZ) connection is a connection on $T^{\vee}_{\reg}$ with values in $\calH_t$. In other words, given a module $M$
over $\calH_t$, we may define a connection $\nabla$ on the trivial bundle over $T^{\vee}_{\reg}$ with fiber $M$.
The connection is defined as follows. Any $\lam\in\grt^*$ defines a vector field on $T^{\vee}$ and thus on
$T^{\vee}_{\reg}$; we shall denote by $d_{\lam}$ the derivative in the direction of this vector field.
In order to define $\nabla$ it is enough to define $\nabla_\lam$ for every $\lam\in \grt^*$, which
is given by the following formula:
\begin{equation}\label{connection}
\nabla_\lam=d_\lam-t\sum\limits_{\alp>0}(\lam,\alp^{\vee}) \frac{q^{\alp^{\vee}}}{1-q^{\alp^{\vee}}}(s_{\alp}-1)-x_\lam
\end{equation}
By specifying $x_{\lam}$'s to their values at some point in $\grt$ we get a family of connections on the trivial bundle with
fiber $M$ over $T^{\vee}_{\reg}$, parametrized by $\grt$. This connection is known to have regular singularities; moreover, it
is clear that it is $W$-equivariant (with respect to the natural action of $W$ on $T^{\vee}_{\reg}$ and {\it trivial} $W$-action
on $M$).

For the sake of completeness, let us compare our form of the connection $\nabla$ with the connection
of \cite{Cher-book}, (1.1.41). In {\em loc. cit.} the author defines a slightly different connection
$\nabla'$ given by
$$
\nabla'_{\lam}=d_\lam-t\sum\limits_{\alp>0}(\lam,\alp^{\vee})
\frac{s_{\alp}}{q^{\alp^{\vee}}-1}-x_\lam.
$$

Let
$$
\del=\prod\limits_{\alp\in R^+}(q^{\alp^{\vee}}-1).
$$
Then it is clear that $\nabla=\del^{-t}\nabla'\del^t$. In other words, our connection $\nabla$ is obtained from Cherednik's
connection $\nabla'$ by a simple gauge transformation. As a byproduct of this remark, we see that $\nabla$ is integrable
(since the same is known about $\nabla'$).

\subsection{The graded nil-Hecke algebra and the Toda connection}\label{hnil}
In this subsection we'll be working with the affine nil-Hecke algebra $\calH_{nil}$,
defined as follows: it is also generated by $x_{\lam}$ for every $\lam\in \grt^*$
 and by elements $\ow$ for $w\in W$ satisfying the following relations:

a')$x_\lam x_\mu=x_\lam x_\mu$ for every $\lam,\mu\in\grt^*$

b')
\begin{equation} \label{nil-Hecke}
\ow_1\ow_2=
\begin{cases}{\ow_1\ow_2}\ \text{if $\ell(w_1w_2)=\ell(w_1)+\ell(w_2)$}\\
0\ \text{otherwise}
\end{cases}
\end{equation}

c')
\begin{equation}
\os_i x_\lam-x_{s_i(\lam)}\os_i=(\alp_i^{\vee}, \lam).
\end{equation}

The algebras $\calH_t$ and $\calH_{nil}$ are related in the following way: let $\tcH_t$ denote the $\CC[t^{-1}]$-subalgebra
of $\calH_t[t^{-1}]$ generated by all the $x_\lam$'s and by the elements $\tilw=t^{-\ell(w)}w$. Then it is clear that
the fiber of $\tcH_t$ at $t=\infty$ is naturally isomorphic to $\calH_{nil}$. In other words we get a $\PP^1$-family $\calH_{\PP^1}$
of associative algebras whose restriction to ${\mathbb A}^1$ is isomorphic to $\calH_t$ and whose fiber at $t=\infty$ is naturally isomorphic
to $\calH_{nil}$. Informally we can say that
$\calH_{nil}$ is the limit of $\calH_t$ when $t\to\infty$; under this limit the elements $x_\lam$ go to themselves and
the limit of $\tilw$ is equal to $\ow$.

Let now $M$ be an $\calH_t$-module which is free over $\CC[t]$ and let $\tilM$ be an $\tcH_t$-submodule of $M[t^{-1}]$
such that $\tilM[t]=M[t^{-1}]$. In other words, $\tilM$ defines an extension of $M$ to a sheaf of modules $M_{\PP^1}$ over
$\calH_{\PP^1}$. The fiber of $\tilM$ at $t=\infty$ is an $\calH_{nil}$-module which we shall
denote by $\oM$. In the future such $\PP^1$-families of $\calH_{\PP^1}$-modules will be of special interest to us.
One example of such family is discussed below.
\subsection{Some modules}
Here we would like to describe explicitly some modules over $\calH_t$ and $\calH_{nil}$, which we are going
to use in the future.

First, we define a module $\ocM$ over $\calH_{nil}$. As a vector space we have
$\ocM=\Sym(\grt^*)=\calO(\grt)$.
The action of $\calH_{nil}$ on $\Sym(\grt^*)$ is described as follows: the action of
$x_{\lam}$ is just given by the multiplication by $\lam$. The action of $\os_i$ (for a simple reflection $s_i$)
is given by
\begin{equation}\label{action-formula-bar}
\os_i(f)=\frac{f-f^{s_i}}{\alp_i},
\end{equation}
where $f^{s_i}(a)=f(s_i(a))$.

Similarly, we can define a module $\calM_t$ over $\calH_t$. Namely, we set
$\calM_t=\Sym(\grt^*)[t]=\calO(\grt\x\CC)$ as a vector space. The action of
$\calH_t$ is defined as follows:
the element $x_{\lam}\in \calH_t$ as before acts by multiplication by $\lam$; $t$ acts in the obvious
way. The action of a simple reflection $s_i\in W$ is defined by
$$
s_i(f(a,t))=f(a,t)-(f(a,t)-f(s_i(a),t))(1-\frac{t}{\alp_i}).
$$
In other words we have
\begin{equation}\label{action-formula}
(1-s_i)f(a,t)=(f(a,t)-f(s_i(a),t))(1-\frac{t}{\alp_i}).
\end{equation}
The verification of the relations of $\calH_t$ is straightforward.

It is also clear that we if we set $\tils_i=t^{-1}s_i$ then
$\lim_{t\to\infty} \tils_i$ is independent of $t$ and is equal to $\os_i$.
Hence when $t\to\infty$ the action of $\tilw=t^{-\ell(w)}w$ becomes independent
of $t$ and goes to the the action of $\ow$ on $\Sym(\grt^*)$.
In other words, if we set $\calM_{\PP^1}$ to be the trivial bundle
over $\PP^1$ with fiber $\Sym(\grt^*)$ then $\calM_{\PP^1}$ becomes a sheaf
of modules over $\calH_{\PP^1}$ whose restriction to $\AA^1$ is $\calM_t$ and whose
fiber at $\infty$ is $\ocM$.

Note that the action of $\Sym(\grt^*)^W$ commutes with the action of $\calH_t$ on $\calM_t$ and with the
action of $\calH_{nil}$ on $\ocM$. For any $\xi\in\grt/W$ let us denote by $\CC_{\xi}$ the
corresponding one-dimensional module over $\Sym(\grt^*)^W$. Then we set
$$
\calM_{\xi,t}=\calM_t\underset{\Sym(\grt^*)^W}\otimes \CC_{\xi};\qquad
\calM_{\xi}=\ocM\underset{\Sym(\grt^*)^W}\otimes \CC_{\xi},
$$
which are modules over $\calH_t$ and $\calH_{nil}$ respectively of dimension $\# W$.

On the other hand, for each $a\in\grt$ let us denote by $\CC_a$ the one-dimensional
$\Sym(\grt^*)$-module supported at $a$. We define
$$
M_{a,t}=\calH_t\underset{\Sym(\grt^*)}\otimes \CC_a.
$$
These are often called {\em the principal series representations} of $\calH_t$.
It is clear that $\dim M_a=\# W$.

The same definitions make sense for $\calH_{nil}$ instead of $\calH_t$. Namely, for any
$a\in \grt$ we define
the principal series representation $M_a$ of $\calH_{nil}$ by
$$
\oM_a=\calH_{nil}\underset{\Sym(\grt^*)}\otimes \CC_a.
$$

The following fact is shown in Theorem 1.2.2 of \cite{Cher-book} in the case of the algebra $\calH_t$;
the same proof as in {\em loc. cit.} works for $\calH_{nil}$.

\begin{prop}\label{principal series}
\begin{enumerate}
\item
There are isomorphisms
$$
\calM_t\simeq \calH_t\underset{\CC[W]}\otimes \CC;\quad
\ocM\simeq \calH_{nil}\underset{\CC[\oW]}\otimes \CC,
$$
where by $\CC[\oW]\subset\calH_{nil}$ we mean the span of all the $\ow$.
\newline

\smallskip
For generic $(a,t)\in\grt\x\CC$ one has:

\smallskip
\item
The modules $M_{a,t}$ and $\oM_a$ are irreducible.
\item
For any $w\in W$ there are isomorphisms
$$
M_{a,t}\simeq M_{w(a),t};\quad \oM_a\simeq \oM_{w(a)}.
$$
\item
Let $\xi$ be the image of $a$ in $\grt/W$. Then there are isomorphisms
$$
M_{a,t}\simeq \calM_{\xi,t};\quad M_a\simeq \ocM_{\xi}.
$$
\end{enumerate}
\end{prop}



\subsection{The Calogero-Moser system}\label{calogero}
The trigonometric Calogero-Moser (or CM for short) quantum integrable system is an embedding
$\eta_{CM,t}$ of the commutative algebra $\Sym(\grt^*)^W$ into the algebra $\calD(T^{\vee}_{\reg})$
of differential operators on $T^{\vee}_{\reg}$, which depends on a parameter $t\in \CC$.
This embedding is characterized by the following properties:

\medskip
CM1) For any $f\in\Sym(\grt^*)^W$ the highest symbol of $\eta_{CM,t}(f)$ is equal to $f$;

CM2) Let us choose a non-degenerate $W$-invariant quadratic form on $\grt^*$; let $C$ be the corresponding
element of $\Sym^2(\grt^*)^W$ and let $\Del$ denote the corresponding Laplacian on $T^{\vee}$ (this is a $W$-invariant
differential operator of order 2 on $T^{\vee}$ with constant coefficients). Then
$$
\eta_{CM,t}(C)=
\Del-t(t-1)\sum\limits_{\alp\in R_+}
\frac{(\alp^{\vee},\alp^{\vee})}{(q^{\alp^{\vee}/2}-q^{-\alp^{\vee}/2})^2}.
$$
Using the map $\eta_{CM,t}$ we may consider $\calD(T^{\vee}_{\reg})[t]$ as a
$\calD(T^{\vee}_{\reg})\otimes \Sym(\grt^*)^W[t]$-module: here $\calD(T^{\vee}_{\reg})$
acts by left multiplication and any  $f\in\Sym(\grt^*)^W$ acts by right multiplication
by $\eta(f)$. We shall call this the Calogero-Moser $\calD$-module and denote it by
$\CM_t$. For any $\xi\in\grt/W$ we denote by $\CM_{\xi,t}$ the specialization of
$\CM_t$ to $\xi$ (i.e. $\calM_{\xi,t}=\calM_t\underset{\Sym(\grt^*)^W}\otimes \CC_{\xi}$).

On the other hand, consider $\calO(T^{\vee}_{\reg})\otimes \calM_t$. Since $\Sym(\grt^*)^W$ acts on
$\calM_t$ by endomorphisms of the $\calH_t$-module structure, by using the connection
$\nabla$ we may view it as a $\calD(T^{\vee}_{\reg})\otimes \Sym(\grt^*)^W[t]$-module.
According to Cherednik (cf. \cite{Cher-many-body} or \cite{Cher-book} Theorem 1.2.11) and Matsuo \cite{Mats} we have
the following result:
\begin{prop}\label{cher-matsuo}
There exists an isomorphism
$$
\calO(T^{\vee}_{\reg})\otimes \calM_t\simeq \CM_t
$$
of $\calD(T^{\vee}_{\reg})\otimes\Sym(\grt^*)^W[t]$-modules.
In particular, for any $\xi\in \grt/W$ there exists an isomorphism
$$
\calO(T^{\vee}_{\reg})\otimes \calM_{\xi,t}\simeq \CM_{\xi,t}
$$
of $\calD(T^{\vee}_{\reg})$-modules.
\end{prop}
\noindent
{\bf Remark.} In fact, in \cite{Mats} Theorem \ref{cher-matsuo} is proved for generic $(a,t)\in\egt$
for $M_{a,t}$ instead
of $\calM_{a,t}$. In \cite{Cher-book} Theorem \ref{cher-matsuo} is proved for $\calM_{a,t}$ (for any $a$)
and it is easy to see that the proof "works in families" (i.e. for $\calM_t$ itself).
\subsection{The classical Calogero-Moser system}The quantum Calogero-Moser system has it quasi-classical analog --
{\em the classical Calogero-Moser system}. This is a $\CC[t]$-linear embedding
$$
\eta_{CM,t}^{cl}:\Sym(\grt^*)^W[t]=\CC[\grg]^{\bfG}\hookrightarrow
\calO(T^*(T^{\vee}_{\reg})\x \CC)=\calO(T^{\vee}_{\reg}\x \grt)[t]
$$
whose image consists of Poisson commuting functions and which satisfies obvious analogs of the conditions CM1, CM2.

\subsection{The monodromy}\label{monodromy}
Let now $\bfH_v$ denote the "usual" affine Hecke algebra; by the definition this is a $\CC[v,v^{-1}]$-algebra,
generated by elements $X_{\lam}$ for $\lam\in \Lam$ and $T_w$ for $w\in W$ subject to the well-known relations.
In particular, $\bfH_v$ is a quotient of the group algebra of $\CC[\hatB_W]$ of the affine braid
group $\hatB_W=\pi_1(T^{\vee}_{\reg}/W)$ associated with the Weyl group $W$; also
the subalgebra of $\bfH_v$ generated by the $X_{\lam}$'s is just $\CC[\Lam]$ which is the same
as the algebra of regular functions on $T$.  The subalgebra of $\bfH_v$ generated by all
the $T_w$ is the {\em finite Hecke algebra} $\bfH_v'$. It is a deformation of
the group algebra $\CC[W]$ and we shall denote by $\mathds{1}$ the corresponding deformation
of the trivial $\CC[W]$-module. For any $z\in T$ we shall denote by $\CC_z$ the corresponding
1-dimensional module over $\CC[\Lam]$.
Abusing notation, we shall sometimes think about $v$ as an element of $\CC^*$ rather than as a formal variable.
Let us set
$$
\bfM_v=\bfH_v\underset{\bfH_v'}\otimes \mathds{1}.
$$
The algebra $\CC[\Lam]^W=\calO(T/W)$ acts on $\bfM_v$ on the right (in fact, this algebra
is the center of $\bfH_v$ and therefore it acts on every $\bfH_v$-module); for any $z\in T/W$
we shall denote by
$\bfM_{z,v}$ the specialization of $\bfM_v$  at $z$. This is a finite-dimensional $\bfH_v$-module of
dimension $\# W$, which is known to be irreducible for generic $z$.

Let us denote $\calH_t-mod^f$ the category of finite-dimensional $\calH_t$-modules; similarly,
let us denote by $\bfH_v-mod^f$ the category of finite-dimensional $\bfH_v$-modules.
Cherednik (cf. e.g. \cite{Cher-book}, Theorem 1.2.8) shows the following:
\begin{prop}\label{cherednik-monodromy}
There exists  a functor
$\calI:\calH_t-mod^f\to \bfH_v-mod^f$, which depends on a choice of a point $q_0\in T^{\vee}_{\reg}/W$,
satisfying the following conditions:
\begin{enumerate}
\item
$\calI(\calM_{\xi,t})=\bfM_{z,v}$ where $z=e^{2\pi i \xi}$ and $v=e^{2\pi i t}$.
\item
For any $M$ in $\calH_t-mod^f$ let us consider the corresponding affine KZ connection, as a connection
over $T^{\vee}_{\reg}/W$ on the trivial bundle with fiber $M$. Then the monodromy representation
of $\CC[\pi_1(T^{\vee}_{\reg}/W, q_0)]$ on $M$ factors through $\bfH_v$ and the resulting $\bfH_v$-module
is $\calI(M)$.
\end{enumerate}
\end{prop}
In other words, one may say that the functor $\calI$ is provided by the monodromy of the affine KZ connection.
\section{Affine Hecke algebras via the Steinberg variety}\label{hecke-geometry}
In this section, we recall how the algebras $\calH_t$ and $\calH_{nil}$ appear geometrically and give a more detailed statement of our main result.
\subsection{The Steinberg variety}
Recall that we have the natural proper (Springer) map $f:T^*\calB\to\calN$.
The Steinberg variety $\St$ is defined as follows:
$$
\St=T^*\calB\underset{\calN}\x T^*\calB.
$$
Explicitly, $\St$ parametrises triples $(\grb_1,\grb_2,x)$ where $\grb_1$ and $\grb_2$ are two Borel
subalgebras in $\grg$ and $x$ is a nilpotent element in $\grb_1\cap \grb_2$.
Alternatively, $\St$ can be defined as the union of the conormal bundles to the $G$-orbits in
$\calB\x\calB$. In particular, $\St$ is equi-dimensional of dimension $2\dim\calB$ and
its irreducible components are parametrised by elements of $W$. The variety $\St$ is endowed with
a natural action of the group $\bG = G\x \CC^*$ where $G$ acts on everything by conjugation and the
multiplicative group $\CC^*$ dilates $x$ (and doesn't change $\grb_1$ and $\grb_2$).
We refer the reader to Section 3.3 of
\cite{chriss-ginz} for further details about $\St$.

\subsection{Borel-Moore homology and Lusztig's construction}\label{bm-lusztig}
Let $H_*^{\bG}(\St)$ denote the $\bG$-equivariant Borel-Moore homology of $\St$.
According to \cite{chriss-ginz} the space is endowed with a natural structure of an associative
algebra. Moreover, this algebra acts naturally on $H_{\bG}(T^*\calB)$ as well as on
$H^*_{Z_n\x\CC^*}(\calS_n)$.

According to \cite{lus-cusp} we have a natural isomorphism\footnote{In fact, in \cite{lus-cusp} this isomorphism is constructed in a much more general situation}
\begin{equation}\label{lusztig}
H_*^{\bG}(\St)\simeq \calH_t.
\end{equation}
Under this isomoprhism $t$ corresponds to the generator of $H^*_{\CC^*}(pt)$ (note that $H_*^{\bG}(\St)$ is a module
over $H^*_{\bG}(pt)$).

Similarly, one can define an associative algebra structure on $H^G_*(\calB\x\calB)$. In this case one has the isomorphism
\begin{equation}\label{lusztig-nil}
H^G_*(\calB\x\calB)\simeq\calH_{nil}.
\end{equation}
The construction of the above isomorphism. Let us just note (cf. \cite{chriss-ginz}) that the algebra $H^*_{\bG}(\St)$
acts naturally on $H_{\bG}(T^*\calB)$ (and more generally on $H^*_{Z_n\x\CC^*}(\tcS_n)$ for any $n\in\calN$) and
the algebra $H^G_*(\calB\x\calB)$ acts on $H^*_G(\calB)$.
\subsection{Explicit construction: geometry}\label{lusztig-explicit}
The isomorphism (\ref{lusztig}) can be described as follows. Namely, according to \cite{lus-cusp}, Section 4, the element
$x_{\lam}\in\calH_t$ corresponds just to the push-forward of the class $D_{\lam}$ under the diagonal embedding
$T^*\calB\hookrightarrow \St$. In particular, the action of $x_{\lam}$ on $H^*_{\bG}(T^*\calB)$ (or, more generally,
on $H^*_{Z_n\x\CC^*}(\tcS_n)$) is given by multiplication by $D_{\lam}$. The description of the image
of $W$ in $H_*^{\bG}(\St)$ is more involved; however one can still describe explicitly the image
of a simple reflection $s_i\in W$ (cf. \cite{lus-cusp}, Section 3). Let us recall this construction.

First of all, for each vertex $i$ of the Dynkin diagram of $G$ let $P_i$ the corresponding sub-minimal
\footnote{Here and and in the sequel the word "subminimal" means that the semi-simple rank of the corresponding
Levi is equal to 1} parabolic and let $\calP_i=G/P_i$ be the variety parametrizing all subgroups of $G$ which are
conjugate to $P_i$. We have the natural projection $p_i:\calB\to\calP_i$ which is a locally trivial $G$-equivariant
$\PP^1$-fibration.
Let $Y_i=\calB\underset{\calP_i}\x\calB$. Then $Y_i$ is a closed subvariety of $\calB\x\calB$ and its
fundamental class $[Y_i]\in H_*^G(\calB\x\calB)$ is equal to $\os_i$. The corresponding operator
on $H^*_G(\calB)$ is just $p_i^*(p_i)_*$.

Similarly, let us now denote by $W_i$ the conormal bundle to $\calB\underset{\calP_i}\x\calB$.
This is a smooth closed $\bG$-invariant subvariety of $\St$ and thus its fundamental class
 $[W_i]$ is a well defined element of $H_*^{\bG}(\St)$ which is equal to $s_i$.
In particular, if we view $W_i$ as a correspondence from
$T^*\calB$ to itself, then its action on $\bG$-equivariant cohomology of $T^*\calB$ is equal to the action
of $s_i-1$.
\subsection{Explicit construction: algebra}
The above actions of $\calH_t$ on $H^*_{\bG}(T^*\calB)$ and of $\calH_{nil}$ on $H^*_G(\calB)$ can be described
explicitly in the following sense. First of all, we have the natural isomorphisms
$$
H^*_G(\calB)\simeq H^*_B(pt)=H^*_T(pt)=\Sym(\grt^*)=\ocM.
$$

Similarly, by using the pull-back map with respect to the projection $T^*\calB\to\calB$ we may identify
$H^*_{\bG}(T^*\calB)$ with $H^*_{\bG}(\calB)=\Sym(\grt^*)[t]=\calM_t$.
Then we have the following
\begin{prop}\label{geometry-algebra}
The above isomorphism $H^*_G(\calB)\simeq \ocM$ is an isomorphism of $\calH_{nil}$-modules.
Similarly, the above isomorphism $H^*_{\bG}(T^*\calB)\simeq \calM_t$ is an isomorphism
of $\calH_t$-modules.
\end{prop}

\subsection{Main result revisited}\label{main-revisited}
With all this context in place, we are now ready to give a reformulation of the main result of this paper.
\begin{thm}\label{main-detailed}
\begin{enumerate}
\item
The operator of quantum multiplication by $D_{\lam}$ in $H^*_{\bfG}(T^*\calB)$ is equal to
$$
x_{\lam} + t \sum\limits_{\alp^{\vee}\in R_+^{\vee}}
(\lambda, \alpha^{\vee}) \frac{q^{\alpha^{\vee}}}{1 - q^{\alpha^{\vee}}} (s_{\alpha}-1) ,
$$
where the action of $\calH_t$ on $H^*_{\bfG}(T^*\calB)$ is the one described above.
\item
The $\bfG$-equivariant quantum connection of $T^*\calB$ is given by Equation \ref{onabla}.
\item
The $\bfG$-equivariant quantum $D$-module of $T^*\calB$ is isomorphic to the Calogero-Moser $\calD$-module
$\CM_t$ (this is an isomorphism of $\calD(T^{\vee}_{\reg})\otimes\Sym(\grt^*)^W[t]$-modules).
\item
The $\bfG$-equivariant quantum cohomology ring of $T^*\calB$ is isomorphic to the ring
of functions on $T^*(T^{\vee}_{\reg})\x \AA^1=T^{\vee}_{\reg}\x \grt\x \AA^1$, where the embedding
$H^*_{\bfG}(pt)=\CC[\grg]^{\bfG}=\Sym(\grt^*)^W[t]\hookrightarrow \calO(T^*(T^{\vee}_{\reg})\x\AA^1)$ is given by the classical Calogero-Moser map $\eta_{CM,t}^{cl}$.
\end{enumerate}
\end{thm}
The proof of the first assertion of Theorem \ref{main-detailed} occupies the next two Sections. Theorem \ref{main-detailed}(2)
is just a reformulation of Theorem \ref{main-detailed}(1). Theorem \ref{main-detailed}(3) follows from
Theorem \ref{main-detailed}(2) by Proposition \ref{cher-matsuo} and Proposition \ref{geometry-algebra}.
Part \ref{main-detailed}(4) was also known to  Nekrasov and Shatashvili \cite{NekSh} for $G=SL(n)$; in the general case it follows
easily from Theorem \ref{main-detailed}(3).
\subsection{Monodromy and derived equivalences}\label{monodromy-derived}
The isomorphism (\ref{lusztig}) has an analog in
$K$-theory (in fact, historically, it was discovered before (\ref{lusztig})). Namely, for a
$\bfG$-scheme $Y$ let us denote by
let $K^{\bfG}(Y)=K_0(\Coh^{\bfG}(Y))\otimes \CC$ the Grothendieck group of the category $\Coh^{\bfG}(Y)$ of
$\bfG$-equivariant coherent sheaves on $Y$. Then $K^{\bfG}(\St)$ has a natural structure of an associative algebra
(defined by convolution) which acts on $K^{\bfG}(T^*\calB)$. Then we have:
\begin{prop}\label{kazhdan-lusztig}\cite{kazhdan-lusztig, chriss-ginz}
\begin{enumerate}
\item
There exists a natural isomorphism
$$
K^{\bfG}(\St)\simeq \bfH_v.
$$
\item
$K^{\bfG}(T^*\calB)$ is isomorphic to $\bfM_v$ as a $K^{\bfG}(\St)=\bfH_v$-module.
\end{enumerate}
\end{prop}
Since $\bfH_v$ is a quotient of the group algebra of the corresponding affine braid group $\hatB_W$, Proposition
\ref{kazhdan-lusztig} provides an action of $\hatB_W$ on $K^{\bfG}(\St)$. This action is categorified
in \cite{Bezr-Riche}. Namely, in {\em loc. cit.} the authors construct an action of $\hatB_W$ on the derived
category $D^b(\Coh^{\bfG}(\St))$, which descends to the above action of $\hatB_W$ on the $K$-theory.
Thus, via Proposition \ref{cherednik-monodromy},
Proposition \ref{geometry-algebra} and Theorem \ref{main-detailed}(2), we can verify the relation between monodromy of the quantum connection and derived auto-equivalences in our situation, as discussed in section \ref{smdr}.

\section{Preliminaries on Gromov-Witten invariants}\label{GWprelim}

In this section, we review some definitions from Gromov-Witten theory and sketch the basic properties of the reduced virtual class.  The latter is a technical construction responsible for many of the nice qualitative properties discussed in the introduction.  Proofs and more detailed discussion of the results here can be found in \cite{okpanhilb, mpk3}.

\subsection{Definitions}

We first clarify the definitions of the equivariant Gromov-Witten invariants
$$
\langle \gamma_1, \dots, \gamma_n \rangle_{0,n,\beta}^{X}
= \int_{[\Mbar_{0,n}(X,\beta)]^{\mathrm{vir}}}
\prod_{k=1}^{n}\ev_{k}^{*}\gamma_k.
$$

In the above expression, the integrand is the virtual class on the moduli space of $n$-pointed stable maps to $X$, which has expected dimension
$$-K_X\cdot\beta + \dim X + n - 3 = 2 \dim \BB + n - 3,$$
and the map $\ev_{k}$ is the evaluation map associated to the $k$-th marked point on the domain curve.

The moduli space of maps is typically noncompact; however there is still a well-defined pushforward morphism
$H^*_{\bG,c}(\Mbar_{0,n}(X,\beta))\to H^*_{\bG}(pt)$ (here $H^*_{\bG,c}$ stands
for the corresponding equivariant cohomology with compact support).
On the other hand, we have the map $H^*_{\bG,c}(\Mbar_{0,n}(X,\beta))\to H^*_{\bG}(\Mbar_{0,n}(X,\beta))$.
This is a map of $H^*_{\bG}$-modules, which
becomes an isomorphism once we tensor everything with the field of fractions of $H^*_{\bG}(pt)$, provided thatthe
$\bT$-fixed locus is compact. Thus we get a well-defined integration on $H^*_{\bG}(\Mbar_{0,n}(X,\beta)$ which
takes values in the field of fractions of
$H_{\bG}^{*}(pt)$.
Explicitly, this integral can be defined using virtual localization.

We further observe that for $\beta \ne 0$,
$$\langle \gamma_{1}, D_{\lambda}, \gamma_{2}\rangle^{X}_{0,3,\beta} =
(D_{\lambda}\cdot \beta) \langle \gamma_{1}, \gamma_{2}\rangle^{X}_{0,2,\beta},$$
by the divisor equation.  As a result, in order to understand the non-classical contribution to divisor operators, it suffices to study the two-point invariants of $X$.

\subsection{Reduced class}\label{reducedclass}

Given a variety $V$ with an everywhere-nondegenerate holomorphic symplectic form $\omega$, it is well-known that the usual non-equivariant virtual fundamental class on $\Mbar_{g,n}(V, \beta)$ vanishes for $\beta \ne 0$.  However, one can correct this phenomenon by modifying the standard obstruction theory so that the virtual dimension is increased by $1$.

We explain this modification for the deformation theory of maps from a fixed nodal curve $C$; applying this construction relative to the moduli stack of nodal curves gives the construction in general.  Let $M_C(V, \beta)$ denote the moduli space of maps from $C$ to $V$ with target homology class $0 \ne \beta \in H_2(V,\ZZ)$.  The standard obstruction theory
 for $M_C(V,\beta)$ is defined by the natural morphism
 \begin{equation}\label{standardmap}
R\pi_{\ast}(\mathrm{ev}^{\ast}T_{V})^{\vee} \rightarrow
L_{M_{C}},
\end{equation}
where $L_{M_{C}}$ denotes the cotangent complex of
$M_{C}(V,\beta)$ and
\begin{gather*}
\mathrm{ev}: C \times M_C(V,\beta) \rightarrow V,\\
\pi: C \times M_C(V, \beta) \rightarrow M_C(V,\beta).
\end{gather*}
are the evaluation and projection maps.

Let $\omega_{\pi}$ denote the relative dualizing
sheaf.    There is a map
$$\mathrm{ev}^{\ast}(T_V)\rightarrow
\omega_{\pi}\otimes(\mathbb{C}\omega)^{\ast},$$
induced by the pairing with the symplectic form and pullback of differentials.
This, in turn,
yields a map of complexes
$$R\pi_{\ast}(\omega_{\pi})^{\vee}\otimes\mathbb{C}\omega \rightarrow
R\pi_{\ast}(\mathrm{ev}^{\ast}(T_V)^{\vee})$$ and the
truncation
$$\iota:\tau_{\leq -1}R\pi_{\ast}(\omega_{\pi})^{\vee}\otimes\mathbb{C}\omega \rightarrow
R\pi_{\ast}(\mathrm{ev}^{\ast}(T_V)^{\vee}).$$  The truncation is a trivial line bundle, although will carry a nontrivial weight in the equivariant setting if the group action acts by a nontrivial character on $\omega$.

There is an induced map
\begin{equation}\label{reducedmap}
C(\iota) \rightarrow L_{M_{C}}
\end{equation}
where $C(\iota)$ is the mapping cone associated to $\iota$.
Moreover, this map
(\ref{reducedmap}) satisfies the necessary
properties of a perfect obstruction theory.  This is known
as the \textit{reduced} obstruction theory and defines
the \textit{reduced} virtual fundamental class associated to $V$.
\subsection{One-parameter families}\label{codimensionone}

Another characterization of reduced virtual classes is given in \cite{mpk3} in terms of one-parameter deformations of the holomorphic symplectic variety $V$.

Suppose we are in the following situation: There is a smooth map
$\pi: \calV \rightarrow B$ where $B$ is a smooth curve, such that $\pi$ is topologically trivial as a fibration and $\calV$ is equipped with a fiber-wise holomorphic symplectic form.  There is a point $b \in B$ such that $V = \calV_{b}$.  Furthermore, assume that we have a class $\beta \in H_{2}(V, \ZZ) \subset H_{2}(\calV, \ZZ)$ satisfying the following conditions:

\begin{itemize}
\item The fiber $V$ is the unique fiber of $\pi$ for which the class $\beta$ is represented by an effective curve.
\item The composition
\begin{equation}\label{embedding}
T_{B,b} \rightarrow H^1(V, T_V) \xrightarrow{\sim} H^{1}(V, \Omega_V) \xrightarrow{\beta\cup} \CC
\end{equation}
is an isomorphism.
\end{itemize}
In the above diagram, the first map is the Kodaira-Spencer map, the second map is the isomorphism induced by the holomorphic symplectic form on $V$, and the last map is given by cup product with the curve class $\beta$.

The following result is proven in Theorem $1$ of \cite{mpk3}
\begin{prop}\label{reduced-versus-virtual}
Given the above hypotheses, the embedding
$$\Mbar_{0,n}(V, \beta) \rightarrow \Mbar_{0,n}(\calV, \beta)$$
is an isomorphism of stacks and we have an
equality of virtual classes
$$[\Mbar_{0,n}(V,\beta)]^{\mathrm{red}} = [\Mbar_{0,n}(\calV, \beta)]^{\mathrm{vir}},$$
where the right-hand side is just the usual virtual fundamental class associated to $\calV$.
\end{prop}

This statement can be easily generalized in many directions.
If we have a fiber-wise group action of $T$ on $\calV$, then the above equality holds for $T$-equivariant virtual classes.  If there are finitely many points on $B$ for which $\beta$ is effective and infinitesimally rigid in the sense of equation (\ref{embedding}), the statement generalizes in the obvious way where the left-hand side involves a summation over these points.

\section{Proof of Theorem \ref{formula}}\label{proofoftheorem}

We now use the results of the last section to give a short proof of Theorem \ref{formula}.

\subsection{Lagrangian correspondences}\label{Lcorr}

We first apply the results of section \ref{reducedclass} to our situation.
Since the holomorphic symplectic form on $X$ has equivariant weight $-t$ with respect to $\bG$, the construction there extends to the equivariant setting.  In particular, the reduced obstruction theory produces a cycle class
$$ [\Mbar_{0,2}(X,\beta)]^{\mathrm{red}} \in H^{BM,\bG}_{2\dim X}(\Mbar_{0,2}(X,\beta)$$
for $\beta \ne 0$ such that
$$[\Mbar_{0,2}(X, \beta)]^{\mathrm{vir}} = t\cdot [\Mbar_{0,2}(X,\beta)]^{\mathrm{red}}.$$
This equality explains the factor of $t$ in front of the nonclassical contribution to Theorem \ref{formula}.

Recall the Springer map
$$f: X \rightarrow \calN.$$
Since $\calN$ is affine, any proper curve on $X$ is contained in a fiber of $f$, and the evaluation map to $X \times X$ factors through the Steinberg variety
$$\phi: \Mbar_{0,2}(X,\beta) \rightarrow \St =  X \times_{\calN} X.$$
The resolution $X \rightarrow \calN$ is semismall, so the irreducible components $Z_{k}$ of $X \times_{\calN} X$ all have dimension $\dim X$.
In particular, we immediately have the following elementary but crucial lemma.
\begin{lemma}\label{correspondences}
\begin{equation}\label{corresp}
\phi_{*}[\Mbar_{0,2}(X,\beta)]^{\mathrm{red}} =
\sum a_k [Z_k] \in H^{BM,\GG}_{2\dim X}(X\times_{\calN}X),
\end{equation}
where $a_k \in \QQ$ are nonequivariant constants.
\end{lemma}
\begin{proof}
This follows immediately from dimension constraints, since we are working with non-localized coefficients.
\end{proof}

Since the correspondence operators defined by $[Z_k]$ give the action of $\ZZ[W]$ on the cohomology of $X$, this explains why the non-classical contribution to $D_{\lambda}$ is expressed in terms of Weyl group operators.   To complete the proof of Theorem \ref{formula}, it remains to match the coefficients $a_{k}$ with the predicted answer.  Since these coefficients are rational numbers, they are unaffected by specializing to $t=0$.  Therefore, we only need to compute the left hand side of
equation \ref{corresp} in $G$-equivariant cohomology.

\subsection{Simultaneous resolution}\label{simultaneous}

We now apply section \ref{codimensionone} to rephrase reduced invariants in terms of one-parameter deformations of $X$; here it is important that we are only working $G$-equivariantly.

We obtain such deformations using the Grothendieck resolution of section \ref{sGroth}:
$$\phi:\widetilde{X} \rightarrow \grt^*.$$  Recall that the fiber over the origin is $X$ and that there is a fiber-preserving $G$-action which cannot be extended to a fiber-preserving $\bG$-action.  In general, for $z \in \grt^*$, the monoid of effective curve classes in the fiber
$\phi^{-1}(z)$ is given by
$$\mathrm{Eff}(\phi^{-1}(z)) = \mathrm{Span}\{\alpha^{\vee} | \alpha \in \Delta_{+},
(\alpha^{\vee},z)=0\}.$$

Consider a generic linear subspace $\phi_{0}: \mathbb{A}^{1} \rightarrow \grt$, chosen to intersect each hyperplane transversely exactly once at the origin; via pullback, we have a smooth map $\calV_0 \rightarrow \mathbb{A}^{1}$ whose fiber over the origin is $X$.

\begin{lemma}
The family $\mathcal{V}_{0}\rightarrow \mathbb{A}^{1}$ satisfies the conditions listed in section \ref{codimensionone}.
\end{lemma}
\begin{proof}
On the set-theoretic level, this is obvious.  The infinitesimal statement follows from the transverse assumption, since the cup product pairing in equation \ref{embedding} can be identified with the pairing of the root lattice with $\grt$.
\end{proof}

Proposition \ref{reduced-versus-virtual} implies that, for $\gamma_{1}, \gamma_{2} \in H_{T}^{*}(X, \CC)$, we have

$$\langle \gamma_{1}, \gamma_{2} \rangle^{X, \mathrm{red}}_{0,2,\beta}
= \langle \gamma_{1}, \gamma_{2}\rangle^{\calV_0}_{0,2,\beta},$$
where the right-hand side again denotes invariants defined via the standard virtual class.

We now deform the variety $\calV_{0}$ and apply deformation invariance, following the approach in \cite{bryan-katz-leung, gwan}.
Let $\calD = \mathbb{A}^{1}$ and consider a family of maps $\phi_{t}:\mathbb{A}^{1} \rightarrow \grt$, $t \in \calD$, so that for $t=0$ we recover $\phi_{0}$ and for any $t \ne 0$ sufficiently near $0$, the image of $\phi_{t}$ intersects each hyperplane $H_{\alpha}$ transversely at distinct points.  We fix $t \ne 0$ with this property, and let $z_{\alpha}$ be the distinct intersection points of $\phi_t$ with $H_{\alpha}$.  The associated smooth family $\mathcal{V}_{t}$ will again satisfy the conditions of section \ref{codimensionone} at each of these intersection points.

\begin{lemma}\label{hyperplanes}
For $t \ne 0$, we have the equality of $G$-equivariant reduced Gromov-Witten invariants
$$\langle \gamma_{1}, \gamma_{2} \rangle^{X, \mathrm{red}}_{0,2,\beta}
= \sum_{\alpha \in \Delta_{+}} \langle \gamma_{1}, \gamma_{2} \rangle^{X_{z_{\alpha}}, \mathrm{red}}_{0,2,\beta}$$
\end{lemma}

\begin{proof}
The left and right sides of the statement can be identified with $\langle \gamma_{1}, \gamma_{2} \rangle ^{\calV_{s}}_{0,2,\beta}$, for $s=0$ and $s=t$ respectively.  These are equal by deformation invariance of the standard virtual class construction applied to the family $\calV \rightarrow \calD$.
Although the associated moduli space of maps relative to $\calD$
$$\Mbar_{0,2}(\calV/\calD, \beta) \rightarrow \calD$$
is not proper over $\calD$, it admits a fiberwise $T$-action whose fixed loci are proper over $\calD$, and the results from \cite{behrend-fantechi} can be applied.
\end{proof}

The arguments in these two sections apply without change to the general setting of the introduction, as discussed in section \ref{rootsubalgebra}.  In particular, we see that the nonclassical terms of divisor operators lie in the subalgebras associated to primitive coroots.

\subsection{Actual calculation}\label{actual}

Lemma \ref{hyperplanes} reduces the argument to understanding the codimension one fibers of the simultaneous resolution.  These have the following description.  Consider a general point $z \in H_{\alpha}$ for
$\alpha \in \Delta_{+}$ and let $M_{z}$ denote its centralizer in $G$; it is easy to see that $M_{z}$ fits into an exact sequence with its center $Z(M_{z})$:

$$1 \rightarrow Z(M_{z}) \rightarrow M_{z} \rightarrow \mathrm{PGL}(2) \rightarrow 1.$$

This sequence induces an action of $M_{z}$ on $T^{*}\PP^1$; the fiber $X_{z}$ is

$$ X_{z} = G \times_{M_{z}} T^{*}\PP^{1},$$

and this identification is compatible with the $G$-action.  This description shows that $X_{z}$ is a fiber bundle over the affine variety  $A_{z} = G/M_{z}$ with fiber $T^{*}\PP^1$; in particular, we see the effective curves are multiples of the zero-section of the fiber, in class $\alpha^{\vee}$.   Let $Y_{z}$ denote the affinization of $X_{z}$, which is just the contraction of this class.

As in the proof of Lemma \ref{correspondences}, evaluation on the moduli space of maps factors through the Steinberg construction:

$$\phi: \Mbar_{0,2}(X_{z},\beta) \rightarrow X_{z} \times_{Y_{z}} X_{z} = X_{z} \cup W_{z}$$
where $W_{z}$ is a $\PP^{1}\times\PP^{1}$-bundle over $A_{z}$.

\begin{lemma}

$$\phi_{*}[\Mbar_{0,2}(X_{z}, d\alpha^{\vee})]^{\mathrm{red}} = \frac{1}{d}[W_{z}]$$

\end{lemma}

\begin{proof}

Let $\Mbar_{0,2}(X_{z}/A_{z}, \beta)$ denote the moduli space of stable maps to $X_{z}$ relative to $A_{z}$, parametrizing a point $a \in A_{z}$ and a stable map to the fiber over  $a$.  Since the map $X_z \rightarrow A_z$ is smooth, this moduli space admits a perfect obstruction theory relative to the map to $A_{z}$.  Furthermore, since we have a fiberwise holomorphic symplectic form on $X_z$, we can construct the reduced obstruction theory relative to $A_z$ and the associated reduced virtual class
$[\Mbar_{0,2}(X_{z}/A_{z}, \beta)]^{\mathrm{red}}.$

We first claim that

$$[\Mbar_{0,2}(X_{z},\beta)]^{red} = [\Mbar_{0,2}(X_{z}/A_{z}, \beta)]^{red}.$$

This follows from a comparison of obstruction theories.  First, for standard virtual classes, there is a map of obstruction spaces associated to maps from a fixed curve $C$
$$H^{1}(C, f^{*}T_{X_z/A_z}) \rightarrow H^{1}(C, f^{*}T_{X_z}).$$ When $C$ is a rational curve, this is an isomorphism.  For reduced classes, we need this map to be compatible with the one-dimensional quotient arising from the holomorphic symplectic form (fiberwise, in the case of $X_{z}/A_{z}$).  This follows from the fact that the holomorphic symplectic form on fibers of
$X_z/A_z$ are the restrictions of the form on $X_z$.

As in Lemma \ref{correspondences}, dimension constraints imply that

$$\phi_{*}[\Mbar_{0,2}(X_{z}, d\alpha^{\vee})]^{\mathrm{red}} = c_{1}[X_{z}] + c_{2}[W_{z}],$$
for $c_{1}, c_{2} \in \QQ$.

We can calculate the coefficients $c_{1},c_{2}$ after restricting to a fiber of $X_z\rightarrow A_z$.  More precisely,

if we pick a point $\iota: a \rightarrow A_{z}$,

$$
\begin{aligned}
\phi_{*}[\Mbar_{0,2}(T^{*}\PP^{1}, d)]^{red} =
\phi_{*}\iota^{!}[\Mbar_{0,2}(X_{z}, d\alpha^{\vee})]^{\mathrm{red}}= \\
\iota^{!}\phi_{*}[\Mbar_{0,2}(X_{z}, d\alpha^{\vee})]^{\mathrm{red}}
= c_{1}[T^{*}\PP^{1}] + c_{2}[\PP^{1}\times\PP^{1}].
\end{aligned}
$$

These coefficients are thus determined from the case $G = \mathrm{SL}_2$, where Theorem \ref{formula} has already been established in this case (\cite{bryan},\cite{gwan}).  The action of the cycle $[\PP^{1}\times\PP^{1}]$ is the operator $s-1$ on $H^{*}(T^{*}\PP^{1},\CC)$, so we see that $c_{1} = 0$ and $c_{2} = 1/d$.

\end{proof}

The proof of Theorem \ref{formula} now follows from the following result at the end of section \ref{lusztig-explicit}.

\begin{lemma}

$$
[W_{z}] = s_{\alpha}-1
$$

\end{lemma}
\section{Shift operators}\label{shift}

We now explain the geometric construction of the shift operators, discussed in section \ref{shift-intro}, that intertwine the monodromy representation of the quantum connection for different values of the equivariant parameters.

Let us choose an integral parameter $s \in \Lambda^\vee\oplus\ZZ$.  We can associate to $s$
a principal $\bT$-bundle over $\PP^1$
$$P(s) = (\bT\times (\CC^2\backslash \{0\})/\CC^*,$$
where $\CC^*$ acts on $\bT$ via the 1-parameter subgroup defined by $s$.
Let
$$\calX(s) = P(s)\times_{\bT} X \rightarrow \PP^1$$ be the associated fibration over $\PP^1$ with fiber $X$.
There exists a unique lifting of the $\CC^*$ action on $\PP^1$ to an action on $\calX(s)$ so that the fiber $\calX(s)_0$ of $0 \in \PP^1$ is fixed pointwise by the action, giving an action of
$\widetilde{\bT}= \bT\times\CC^*$ on $\calX(s)$.  Let $z$ denote the equivariant parameter associated to the extra torus factor.

Given a weight $\lambda \in \Lambda$, there is again an associated equivariant line bundle $D_{\lambda,s}$ on $\calX(s)$, compatible with restriction to fibers.  Given $\beta \in H_{2}(\calX(s),\ZZ)$, we can use these line bundles to assign a natural monomial function $q^{\beta}$ on the dual torus $T^{\vee}$.

We will only consider curve classes $\beta \in H_{2}(\calX(s),\ZZ)$ such that
$$\pi_*\beta = [\PP^1],$$ and the corresponding moduli spaces of stable maps
$$\Mbar_{0,2}(\calX(s),\beta).$$  Let
$$\iota_{0}, \iota_{\infty}: X \hookrightarrow \calX(s)$$
denote the inclusions of the fiber over $0$ and $\infty$ respectively.
Given $\gamma_{1}, \gamma_{2} \in H_{\bT}^{*}(X),$
the two-pointed Gromov-Witten invariant
$$\langle \gamma_{1}, \gamma_{2}\rangle^{\calX(s)}_{0,2,\beta} =
\int_{[\Mbar_{0,2}(\calX(s),\beta)]^{\mathrm{vir}}} \ev_1^*(\iota_{0,*}\gamma_{1})\cup
\ev_2^*(\iota_{\infty,*}\gamma_{2}) \in \mathrm{Frac} H_{\widetilde{\bT}}^*(X).$$
In order to make sense of the above expression, for each fiber over $0$ and $\infty$, we use the identification $H_{\widetilde{\bT}}^*(X) = H_{\bT}^{*}(X) \otimes \CC[z]$
to lift $\gamma_1$ and $\gamma_2$ to $\widetilde{\bT}$-equivariant cohomology.

The corresponding shift operator
$$\bS(s) \in \End(H_{\bT}^{*}(X)) \otimes \CC[[\Lambda]]$$
is then defined via the pairing
$$\langle \gamma_{1}| \bS(s)| \gamma_{2}\rangle
= \sum_{\substack{\beta \in H_{2}(\calX(s),\ZZ)\\ \pi_*\beta = [\PP^1]}}
q^{\beta}\left( \langle \gamma_{1}, \gamma_{2}\rangle^{\calX(s)}_{0,2,\beta} \right)\mid_{z=1}.
$$

The intertwiner property of $\bS(s)$ can be seen as follows.  We first describe the contribution of the residues associated to classical multiplication.
Given $x \in X^{T}$ associated to the Weyl group element $w_x \in W$, let
$$\phi_x= (w_x, 0): \mathfrak{t}^* \rightarrow \mathfrak{t}^{*}\oplus \CC = (\mathrm{Lie}(\bT))^{*}$$ denote the associated linear map, which we can view as a multi-valued
function on $T^{\vee}$, depending on a parameter $a \in \mathrm{Lie}(\bT)$.
We define the operator $q^{M}$ on $H_{\bT}^{*}(X)$ which is diagonal in the fixed-point basis, with diagonal entries given by the function $\phi_x$.

In the neighborhood of the base point
$$q^{\alpha_{1}^{\vee}}= q^{\alpha_{2}^{\vee}} = \cdots = 0,$$
there exists a unique fundamental solution of the quantum differential equation, normalized to have the form
$$\Psi(q,a)= q^{M}(\mathrm{Id} + \mathrm{O}(q^{\alpha_{i}^{\vee}})).$$
Here, we have made explicit the dependence of $\Psi$ on the equivariant parameters
$a \in \mathrm{Lie}(\bT)$.
There exists a geometric expression for $\Psi(q,a)$ in terms of descendent invariants (see, for example, \cite{CoxKatz}).

We also define the operator $\Delta(s)$ on $H_{\bT}^{*}(X)$, so that it is diagonal in the fixed-point basis, with diagonal entries given by the ratio
$$\Delta(s)_{p,p} = \prod_{\substack{\text{tangent weights } \rho\\ \text{at }x \in X}} \frac{\Gamma(\rho + s(\rho))}{\Gamma(\rho+1)}.$$

It follows from equivariant localization with respect to $\widetilde{\bT}$
that $\bS(s)$ can be factored in terms of the fundamental solution and $\Delta$.
\begin{equation}
\bS(s) = \Psi(q,a) \circ \Delta(s)\circ \Psi^{-1}(q, a+s).
\end{equation}

From here, we see immediately that the operators $\bS(s)$ satisfy the following intertwiner and composition identities:
\begin{align}
\nabla(a) \, \bS(s) &= \bS(s) \, \nabla(a+s) \\
\bS(s_1+s_2) &= \bS(s_1)\circ\bS(s_2)|_{a=a+s_1}
\end{align}

In order to compare the monodromy representation after integral parameter shifts
$$a \mapsto a+s,$$
we use the following proposition.
\begin{prop}\label{intertwiner-rational}
The matrix entries of $\bS(s)$ are rational functions on $\CC(T^{\vee})$:
$$\bS(s) \in \End(H_{\bT}^{*}(X)) \otimes \CC(T^{\vee}).$$
\end{prop}
\begin{proof}

We will only sketch the proof here.
In the case of $\mathrm{Hilb}(\CC^{2})$ and higher-rank framed sheaves, a version of this argument can be found in
\cite{moop} and \cite{MO}.
We refer the reader to these papers for a more detailed exposition.

Using the composition identity for $\bS(s)$, it suffices to prove the proposition for
$$s = (\delta,0), (0,1) \in \Lambda^{\vee}\oplus \ZZ,$$
where $\delta$ is a fundamental weight of $T^{\vee}$.
The virtual dimension of the moduli spaces $\Mbar_{0,0}(\calX(s),\beta)$ are, respectively,
$2\dim \BB$ and $\dim \BB$.

In the first case, if we choose insertions over $0$ and $\infty$ from a basis of Schubert conormal Lagrangians in $X$, the relevant two-point invariants will be degree $0$.  Moreover, there exists a choice of basis for which the space of stable maps meeting these cycles is proper.  As a consequence, the associated matrix entries are degree $0$ equivariant polynomials, thus constant.  If we specialize to $t=0$, we can deform $\calX(s)$ so that it has affine fibers, and only sections of $\pi$ contribute.  In particular, the answer can be calculated directly (e.g. using the case of $\BB = \PP^1$).

The second case is more involved.  We consider the subspace $V$ of vectors
$v\in H_{\bT}(X)\otimes \CC(T^{\vee})$ for which
$\bS(s)v$ has rational function entries.  We require two claims about $V$:
\begin{itemize}
\item $V$ contains the identity element
$1 \in H_{\bT}^*(X)$
\item $V$ is closed under application of the differential operators
$\nabla_{\lambda}.$
\end{itemize}
The first claim follows again via equivariant specialization.  By construction, $\Mbar_{0,2}(\calX(s),\beta)$ is proper - any section of $\calX(s) \rightarrow\PP^1$ is forced to lie inside $\BB\times\PP^1$.  Therefore, when expressed in a basis of Schubert conormals, every entry in $\bS(s)\circ 1$ is a nonequivariant constant, so the claim can be checked after setting $t=0$.  If we apply virtual localization, we require the specialization of the fundamental solution of the quantum connection when $t=0,1$.  However, it follows from the form of the Calogero-Moser system given in section \ref{calogero} that the connection becomes essentially trivial at these values, yielding the result.
The second claim follows immediately from the intertwiner property.  Since classical cohomology is generated by divisors, we then conclude from these two claims that $V = H_{\bT}(X)\otimes \CC(T^{\vee}).$
\end{proof}

\begin{cor}
For $a \in \mathrm{Lie}(\bT)$ away from a discriminant locus (depending on $s$), the monodromy representation of the quantum differential equation $\nabla(a)$ is isomorphic to the monodromy representation
for $\nabla(a+s)$
\end{cor}
Notice that the intertwiners $\bS(s)$ may have acquire singularities as a function of $a$, accounting for the discriminant locus in the statement of the corollary.  It follows from irreducibility of the monodromy representation that the geometric intertwiners agree with the Opdam shift operators, up to a global equivariant constant.

If we consider a general equivariant symplectic resolution as in section \ref{geom_sympl}, we expect analogous statements to hold; however the arguments given here only partially extend.  While the definition of the geometric intertwiner is completely formal, one needs to check that the proof in Proposition \ref{intertwiner-rational} holds in any given geometry.  In particular, the simplification of the quantum connection for $t=1$ and the generation of the quantum ring by divisors are the key statements that need to be established.

\section{The limit $t\to \infty$: algebra}\label{limit-algebra}
The purpose of the next two sections is to explain how Theorem \ref{formula} allows us
to compute the quantum cohomology of $\BB$.

\medskip
\noindent
\subsection{The connection $\onabla$}
Similarly to the affine KZ connection $\nabla$ we can define the Toda connection $\onabla$, which is a connection
on $T^{\vee}$
with values in $\calH_{nil}$.  More precisely, given a module $\oM$ over $\calH_{nil}$
we have a connection on the trivial bundle over $T^{\vee}$ with fiber $M$ over $T^{\vee}$
defined as follows. First, let us recall the subset  $R_+'$  of $R_+$ defined in section \ref{int-limit}.
Then we define
\begin{equation}\label{onabla}
\onabla_\lam=d_\lam-\sum\limits_{\alp\in R_+'} (\lam,\alp^{\vee})q^{\alp^{\vee}}\os_{\alp}-x_\lam.
\end{equation}
We are going to show that $\onabla$ is integrable a little later. We shall refer to $\onabla$ as
{\em the Toda connection}; the reason for this name is explained in section \ref{toda}.


We now claim that under the above circumstances the connection
$\onabla$ is "the limit of $\nabla$ as $t\to\infty$" in the appropriate sense.
More precisely, we have the following
\begin{proposition}\label{prop-limit}
The limit
$$
\lim\limits_{t\to\infty}\,  t^{2\rho}\nabla\,  t^{-2\rho}
$$
exists and is equal to $\onabla$ (which is a connection on $T^{\vee}$ with values in the trivial bundle with fiber $\oM$).
In other words, $\onabla$ is the $t\to\infty$-limit of $\nabla$ if we make the change of variables
$q\to t^{-2\rho}q$.
\end{proposition}

\noindent
{\bf Remark.}
Note that the singularities of the connection $\nabla$ on $T^{\vee}$ disappear when $t\to\infty$; however one can check that
in general the connection $\onabla$ will have irregular singularities at $\infty$. Also, the $W$-equivariance of
$\nabla$ is lost under the above limit procedure (this has to do with the fact that our change of variables
$q\to t^{-2\rho}q$ is not $W$-equivariant).
\begin{proof}
>From equation \ref{connection} and from the definition of $\tilw$ we get
$$
t^{2\rho}\, \nabla_\lam\,  t^{-2\rho} =d_\lam+\sum\limits_{\alp>0}
\frac{t^{-(2\rho,\alp^{\vee})+1}q^{\alp^{\vee}}}{1-t^{-(2\rho,\alp^{\vee})}q^{\alp^{\vee}}}
(1-t^{\ell(s_{\alp})}
\tils_{\alp})+x_{\lam}.
$$
Thus, in order to prove Proposition \ref{prop-limit}, it is enough to check the validity of the following
\begin{lemma}
\begin{enumerate}
\item
For any $\alp\in R_+$ we have
$$
\ell(s_{\alp})\leq (2\rho,\alp^{\vee})-1.
$$
\item
The above inequality is an equality if and only if $\alp\in R_+'$.
\end{enumerate}
\end{lemma}
\begin{proof}
This lemma is essentially proved in Section 1 of \cite{jut}. More precisely, it is shown in \cite{jut}(Lemma 1.3)
that $\ell(s_{\alp})\leq (2\rho,\alp^{\vee})-1$ when $\alp$ is long and that
$\ell(s_{\alp})=(2\rho^{\vee},\alp)-1$ when $\alp$ is short. Thus it is enough to show that for
any short positive root $\alp$ we have
\begin{equation}\label{inequality}
(\rho^{\vee},\alp)\leq (\rho,\alp^{\vee}),
\end{equation}
where the equality holds if and only if $\alp\in R_+'$. Note that if $\alp$ is short, then $\alp^{\vee}$ is long.
Let $\Pi^{\vee}_{\lg}$ denote the set of simple long coroots and let $\Pi^{\vee}_{\sh}$ denote the set of simple short
coroots. Let
$$
\alp^{\vee}=\sum\limits_{\beta^{\vee}\in\Pi^{\vee}_{\lg}} a_{\beta}\beta^{\vee}+
\sum\limits_{\gamma^{\vee}\in\Pi^{\vee}_{\sh}} a_{\gam}\gam^{\vee}.
$$
Let us choose a $W$-invariant inner product on $\grt^*$ satisfying $(\alp,\alp)=1$ for every short
root $\alp$ and set
$$
r=\max\limits_{\alp\in R_+} (\alp,\alp).
$$
Then we have $(\rho,\alp^{\vee})=\sum a_{\beta}+\sum a_{\gam}$ and according to Section 1.2 of \cite{jut}
we have
$$
(\rho^{\vee},\alp)=\sum\limits_{\beta^{\vee}\in\Pi^{\vee}_{\lg}} a_{\beta}+
\frac{1}{r}\sum\limits_{\gamma^{\vee}\in\Pi^{\vee}_{\sh}} a_{\gam},
$$
which implies (\ref{inequality}).
\end{proof}
\end{proof}
\begin{cor}\label{cor-int}
The connection $\onabla$ is integrable.
\end{cor}
\begin{proof}
This follows immediately from Proposition \ref{prop-limit} and from the fact that $\nabla$ is integrable.
\end{proof}

\subsection{The Toda system}\label{toda}
Similarly to section \ref{calogero} one can define the quantum Toda integrable system, which is now a map
$\eta_T:\Sym(\grt^*)^W\to \calD(T^{\vee})$. The map $\eta_T$ is characterized by the following
properties:

\medskip
T1) For any $f\in\Sym(\grt^*)^W$ the highest symbol of $\eta_{T}(f)$ is equal to $f$;

T2) For any non-degenerate $W$-invariant quadratic form on $\grt^*$ as above one has
$$
\eta_T(C)=\Del-\sum\limits_{\alp\in \Pi}{(\alp^{\vee},\alp^{\vee})}q^{\alp^{\vee}}.
$$

\medskip
\noindent

It is known (cf. Section 7 of \cite{etingof}) that for any $f\in \Sym(\grt^*)^W$ one has
\begin{equation}\label{etingof}
\lim\limits_{t\to\infty}t^{2\rho}\eta_{CM,t}(f)t^{-2\rho}=\eta_T(f).
\end{equation}

Using the map $\eta_T$ we may view $\calD(T^{\vee})$ as a
$\calD(T^{\vee})\otimes \Sym(\grt^*)^W$-module, which we shall denote by $\calT$.
We shall call $\calT$ {\em the Toda $\calD$-module}.

We now want to present an analog of Theorem \ref{cher-matsuo} in the case
when the Calogero-Moser system is replaced by the Toda system. Namely,
by using the connection $\onabla$ we may view $\calO(T^{\vee})\otimes \ocM$ as a
$\calD(T^{\vee})\otimes \Sym(\grt^*)^W$-module.
We now claim the following analog of the Cherednik-Matsuo theorem for the
Toda lattice:
\begin{thm}\label{cher-matsuo-toda}
There exists an isomorphism
$$
\calO(T^{\vee})\otimes \ocM\simeq \calT
$$
of $\calD(T^{\vee})\otimes \Sym(\grt^*)^W$-modules.
\end{thm}
\noindent
{\bf Remark.} It is tempting to say that Theorem \ref{cher-matsuo-toda} follows from Theorem
\ref{cher-matsuo} by taking $t\to\infty$ limit and using Proposition \ref{prop-limit}
together with (\ref{etingof}). However, {\em a priori} it is not clear why the limiting procedure
in Proposition \ref{prop-limit} is the same as in (\ref{etingof}). Hence we are going to give
an independent proof of Theorem \ref{cher-matsuo-toda}.

\begin{proof}
Let us define a map $\iota:\calT=\calD(T^{\vee})\to \calO(T^{\vee})\otimes \ocM=\calO(T^{\vee})\otimes
\Sym(\grt^*)$ which send every $d\in\calD(T^{\vee})$ to $d(1\otimes 1)$ (let us recall that the action
of $\calD(T^{\vee})$ on $\calO(T^{\vee})\otimes \Sym(\grt^*)^W$ is given by $\onabla$).
\begin{lem}
The map $\iota$ is an isomorphism of vector spaces.
\end{lem}
\begin{proof}
The module $\calT=\calD(T^{\vee})$ is endowed with a natural filtration by order of a differential operator.
Similarly, $\calO(T^{\vee})\otimes \Sym(\grt^*)$ is endowed with a filtration coming from
the natural filtration on $\Sym(\grt^*)$ (by degree of a polynomial). From the definition of $\iota$ it is clear,
that it is compatible with the above filtration it defines between the associated graded spaces.
This shows that $\iota$ is an isomorphism itself.

\end{proof}
On the other hand, it is obvious that $\iota$ is a morphism of $\calD(T^{\vee})$-modules. Hence,
in order to prove Theorem \ref{cher-matsuo-toda} it is enough to verify that it is a
morphism of $\Sym(\grt^*)^W$-modules. We claim that it is enough to check this for quadratic elements of
$\Sym(\grt^*)^W$. Indeed, applying $\iota^{-1}$ to the action of $\Sym(\grt^*)^W$ on
$\calO(T^{\vee})\otimes \Sym(\grt^*)$ we get an action of $\Sym(\grt^*)^W$ on $\calT=\calD(T^{\vee})$;
this action commutes with the left $\calD(T^{\vee})$-action and thus comes from a homomorphism
$\eta':\Sym(\grt^*)^W\to \calD(T^{\vee})$. We need to show that $\eta'=\eta_T$.
However, $\eta'$ obviously satisfies property T1.
Thus in order to show property the equality $\eta'=\eta_T$ we need to check that $\eta'$ satisfies T2, which
exactly means that $\iota$ commutes with quadratic elements in $\Sym(\grt^*)^W$.

Thus we are reduced to checking that for any $C$ as in T2 we have
\begin{equation}\label{desire}
\eta(C)(1\otimes 1)=1\otimes C.
\end{equation}
Let $\lam_1,\cdots,\lam_{\ell}$ be a basis of $\grt^*$ and let $\mu_1,\cdots,\mu_{\ell}$ be the dual
basis with respect to the quadratic form corresponding to $C$. Then
$\onabla_{\lam_i}(1\otimes 1)=-1\otimes \lam_i$. Thus
$$
\begin{aligned}
&\Del(1\otimes 1)=\sum\limits_{i=1}^{\ell}\onabla_{\mu_i}\onabla_{\lam_i}(1\otimes 1)=
-\sum\limits_{i=1}^{\ell} \onabla_{\mu_i}(1\otimes \lam_i)=
\sum\limits_{i=1}^{\ell}
(1\otimes \mu_i\lam_i+\\
&\sum\limits_{\alp\in R_+'}(\mu_i,\alp^{\vee})q^{\alp^{\vee}}(1\otimes \os_{\alp}(\lam_i)))=
1\otimes C+\sum\limits_{\alp\in R_+'}q^{\alp^{\vee}}
\sum\limits_{i=1}^{\ell}(\mu_i,\alp^{\vee})(1\otimes \os_{\alp}(\lam_i))
\end{aligned}
$$
Note that $\os_{\alp}(\lam_i)=0$ if $\alp$ is not a simple root; for $\alp\in\Pi$ we have
$\os_{\alp}(\lam_i)=(\lam_i,\alp^{\vee})$.
Thus
$$
\sum\limits_{\alp\in R_+'}\sum\limits_{i=1}^{\ell}(\mu_i,\alp^{\vee})
(q^{\alp^{\vee}}\otimes \os_{\alp}(\lam_i))=
\sum\limits_{\alp\in \Pi}q^{\alp^{\vee}}\sum\limits_{i=1}^{\ell}(\mu_i,\alp^{\vee})(\lam_i,\alp^{\vee})=
\sum\limits_{\alp\in\Pi}(\alp^{\vee},\alp^{\vee})q^{\alp^{\vee}}(1\otimes 1),
$$
which implies (\ref{desire}).
\end{proof}

\section{The limit $t\to \infty$: geometry}\label{limit-geometry}

In this section we explain how
we can recover the quantum cohomology of $\BB$ by studying the $t \rightarrow \infty$ limit of the quantum cohomology
of $X=T^*\BB$.
Using the projection $\pi:X \rightarrow \BB$, we can identify
$$
H_{\TT}^{*}(X, \CC) = H_{T}^{*}(\BB,\CC)[t],
$$
and view the divisor operators $D_{\lambda}^{X}$ as acting on the latter space:
$$
D_{\lambda}^{X}(t, q) \in \mathrm{End}(H_T^{*}(\BB,\CC))[t][[q^{\beta}]].
$$
Let
$D_{\lambda}^{X}(t, t^{-c_{1}}q)$ be the operators obtained by the substitution
$$q^{\beta} \mapsto t^{-c_{1}(\BB)\cdot\beta}q^{\beta}.$$
We want to compare these operators with the divisor operators $D_{\lambda}^{\BB}(q)$ arising from the quantum cohomology of $\BB$; since the $\TT$-action on $\BB$ factors through $T$, there is no $t$-dependence here.

\begin{prop}\label{prop-limit-geometry}
$$
D_{\lambda}^{\BB}(q) = \lim\limits_{t\rightarrow \infty} D_{\lambda}^{X}(t, t^{-c_{1}}q).
$$
\end{prop}
\begin{proof}
Matrix elements of $D_{\lambda}^{X}$ are obtained from invariants of the form
$$\langle \gamma_{1}\cdot [\BB], D_{\lambda}, \gamma_{2}\rangle^{X}_{0,3,\beta},$$
where $\gamma_{1},\gamma_{2}\in H_{T}^{*}(\BB,\CC)$ and
$[\BB]$ denotes the class of the zero-section.
Define the moduli space of maps
$$\Mbar_{0,3}(X,p_1,\beta) = \{(C,f) \in \Mbar_{0,3}(X,\beta)| f(p_{1}) \in \BB\subset X\}.$$
It is a standard fact that $\Mbar_{0,3}(X,p,\beta)$ admits a $\bT$-equivariant virtual class for which
$$\iota_{*}[\Mbar_{0,3}(X,p_1,\beta)]^{\mathrm{vir}} = \mathrm{ev}_{1}^{*}[\BB] \cdot [\Mbar_{0,3}(X,\beta)]^{\mathrm{vir}}.$$

Moreover, since $T\BB$ is nef, $f(p_1) \in \BB$ implies that $f$ factors through $\BB$, and we have
$$\Mbar_{0,3}(X,p_1,\beta) = \Mbar_{0,3}(\BB,\beta)$$
as stacks, equipped with two different obstruction theories.

Consider the universal family
$$
\xymatrix{
\mathcal{C} \ar[r]^{f}\ar[d]^{\pi} & \BB \\
\Mbar_{0,3}(\BB,\beta) \ar@/^/[u]^{\sigma_{1}}
}
$$
where $\sigma_1$ is the section associated to the first marked point, with image $\Sigma_1 \subset \mathcal{C}$.
We can define a sheaf $\mathcal{K}$ on $\mathcal{C}$ by the short exact sequence
$$0 \rightarrow \mathcal{K} \rightarrow f^{*}T^{*}\BB \rightarrow f^{*}T^{*}\BB|_{\Sigma_1}\rightarrow 0,$$
where the second map is the restriction map.  It is easy to see that $R^{0}\pi_{*}\mathcal{K}=0$
and that
$$E = R^1\pi_{*}\mathcal{K}$$
is a $\bT$-equivariant vector bundle on
$\Mbar_{0,3}(\BB,\beta)$
with rank $c_{1}(\BB)\cdot\beta$.

The relative obstruction theories are then related by the triangle
$$E^{\vee}[1] \rightarrow E_{\Mbar(X,p_1,\beta)/\mathfrak{M}_{0,3}} \rightarrow E_{\Mbar(\BB,\beta)/\mathfrak{M}_{0,3}},$$
and we have
$$ [\Mbar_{0,3}(X,p_1,\beta)]^{\mathrm{vir}} = c_{t}(E)\cap [\Mbar_{0,3}(\BB,\beta)]^{\mathrm{vir}}$$
where
$$c_{t}(E) = t^{\mathrm{rk}(E)} + c_{1}(E)t^{\mathrm{rk}(E)-1} + \dots + c_{\mathrm{rk}(E)}(E)$$
is the $\bT$-equivariant Euler class.
This yields
$$\langle \gamma_{1}\cdot [\BB], D_{\lambda}, \gamma_{2}\rangle^{X}_{0,3,\beta}
= \langle c_{t}(E), \gamma_{1}, D_{\lambda}, \gamma_{2}\rangle^{\BB}_{0,3,\beta}.$$
If we want to recover the term corresponding to the operator $D_{\lambda}^{\BB}$, we can weight this invariant by $t^{-\mathrm{rk}(E)}= t^{-c_{1}(\BB)\cdot\beta}$ and take the limit as $t \rightarrow\infty$.
This same argument works for all multiplication operators.
\end{proof}
Let us now recall that $c_1(\calB)=-2\rho$. Combining this with Proposition \ref{prop-limit-geometry}, Proposition \ref{prop-limit},
and Theorem \ref{cher-matsuo-toda}
we obtain a new proof of the following result:
\begin{thm}\label{usual flags}
\begin{enumerate}
\item
The operator of quantum multiplication by $D_{\lam}$ in $H^*_G(\calB)$ is equal to
$$
x_{\lam}+\sum\limits_{\alp\in R_+'}(\lam,\alp^{\vee})q^{\alp^{\vee}}s_{\alp}.
$$
Here $x_{\lam}$ is the classical multiplication by $D_{\lam}$.
\item
The $G$-equivariant quantum connection of $\calB$ is given by Equation \ref{onabla}.
\item
The $G$-equivariant quantum $D$-module of $\calB$ is isomorphic to the Toda $\calD$-module
$\calT$ (as a $\calD(T^{\vee})\otimes\Sym(\grt^*)^W$-module).
\end{enumerate}
\end{thm}
These statements are already well-known in the literature
The first follows from Theorem 6.4 in \cite{mih}.
The second statement is of course just a reformulation of the first, and the third statement follows from the main result of \cite{kim}, although via a different argument.

\section{Slodowy slices}\label{slodowy}

In this section, we prove Theorem \ref{formula-slodowy}.

Given $n \in \calN$, we can associate an $\mathfrak{sl}_2$-triple $(e=n, h, f)$ and consider the affine subspace $\mathfrak{v}_{n} = \{ n + \mathrm{Ker}\mathrm{ad} f \} \subset \grt$.  The Slodowy slice resolution $\tcS_n$ is a local complete intersection subvariety of $X$ defined by the Cartesian diagram

\begin{equation}\label{slodowy1}
\xymatrix{
\tcS_n \ar@{^{(}->}[r]^{\iota}\ar[d]^{g'} & X \ar[d]^{g} \\
\mathfrak{v}_n  \ar@{^{(}->}[r]^{\iota}& \grt\\
}
\end{equation}

Let $\TT' \subset \TT$ be the subtorus which preserves $\tcS_{n}$.
Assume the natural map $H_2(\tcS_n,\ZZ)\rightarrow H_2(X,\ZZ)$ is an isomorphism.
Given $\beta \in H_2(\tcS_n,\ZZ)$, since $\mathfrak{v}_n$ and $\grt$ are affine, the moduli space of stable maps $\Mbar_{0,2}(\tcS_n,\beta)$ fits in the Cartesian diagram

\begin{equation*}
\xymatrix{
\Mbar_{0,2}(\tcS_n,\beta)\ar[r]^{\iota}\ar[d]^{h} & \Mbar_{0,2}(X,\beta) \ar[d]^{h} \\
\mathfrak{v}_n  \ar@{^{(}->}[r]^{\iota}& \grt\\
}
\end{equation*}

In order to prove Theorem \ref{formula-slodowy}, we will show the following.

\begin{prop}  There is an equality of $\TT'$-equivariant virtual classes
$$\iota^{!}[\Mbar_{0,2}(\tcS_n,\beta)]^{\mathrm{vir}} = [\Mbar_{0,2}(X,\beta)]^{\mathrm{vir}}$$
\end{prop}
\begin{proof}
As in section \ref{reducedclass}, we will prove the analogous statement for the moduli space
$M_C(\tcS_n,\beta)$ of maps from a fixed genus $0$ nodal curve $C$; the full statement will then follow by applying this statement relative to the moduli stack of nodal curves.

In order to prove the pullback of virtual classes, it is sufficient to prove that the obstruction theories on
$M_C(\tcS_n,\beta)$ and $M_C(X,\beta)$ are compatible in the sense of Theorem $5.10$ of
\cite{behrend-fantechi}.  To show this, we need to construct a morphism of exact triangles
\begin{equation*}
\xymatrix{
\iota^*E_{M_C(X,\beta)} \ar[r]^{\phi}\ar[d] & E_{M_C(\tcS_{n},\beta)} \ar[r]^{\psi} \ar[d] & h^{*}L_{\mathfrak{v}_{n}/\grt}\ar[r]^{\chi}\ar[d] &\iota^*E_{M_C(X,\beta)}[1]\ar[d]
\\
L_{M_C(X,\beta)} \ar[r] & L_{M_C(\tcS_{n},\beta)} \ar[r] & L_{M_C(\tcS_{n},\beta)/M_C(X,\beta)}\ar[r]
&L_{M_C(X,\beta)}[1]\\
}
\end{equation*}
Here, the bottom row is the standard triangle associated to the cotangent complex of a morphism, and the vertical maps are the data of an obstruction theory.

The top row will be induced by the short exact sequence
\begin{equation}\label{normalsequence}
0 \rightarrow T_{\tcS_{n}} \rightarrow \iota^{*}T_{X} \rightarrow g^{*}N_{\mathfrak{v}_{n}/\grt}  \rightarrow 0.
\end{equation}
To see this, recall that, in the notation of section \ref{reducedclass}, we have
$$E_{M_C(V,\beta)} = R\pi_{\ast}(\mathrm{ev}^{\ast}T_{V})^{\vee}$$
for $V = X$ and  $\tcS_{n}$.  So if we apply
$R\pi_{\ast}(\mathrm{ev}^{\ast}( - )^{\vee}$ to the sequence (\ref{normalsequence})
we get the triangle
$$
\iota^*E_{M_C(X,\beta)} \rightarrow E_{M_C(\tcS_{n},\beta)} \rightarrow
R\pi_{\ast}(\mathrm{ev}^{\ast}g^*N_{\mathfrak{v}_{n}/\grt})^{\vee}[1]
\rightarrow \iota^*E_{M_C(X,\beta)}[1].$$
Since the genus of $C$ is $0$, we have
$R\pi_{\ast}(\mathrm{ev}^{\ast}g^*N_{\mathfrak{v}_{n}/\grt})^{\vee}[1] = h^{*}L_{\mathfrak{v}_{n}/\grt},$
and the maps obviously commute with the vertical maps associated to each obstruction theory.

Since the obstruction theories associated to $X$ and $\tcS_{n}$ are compatible, Theorem $5.10$ in \cite{behrend-fantechi} then immediately gives an isomorphism of virtual classes after pullback.
\end{proof}

We use this proposition to give a short proof of Theorem \ref{formula-slodowy}.
\begin{proof}
As in the proof of Theorem \ref{formula}, the identification of classical multiplication by $D_\lambda$ follows from the definition of the Hecke algebra action on $H_{\TT'}^*(\tcS_n)$.  For the nonclassical terms, it suffices to calculate the two-point Gromov-Witten invariants of $\tcS_n$.

Consider the Cartesian diagram induced by the evaluation maps to the Steinberg variety:
\begin{equation*}
\xymatrix{
\Mbar_{0,2}(\tcS_n,\beta)\ar[r]^{\iota}\ar[d]^{\phi} & \Mbar_{0,2}(X,\beta) \ar[d]^{\phi} \\
\tcS_n \times_{\calS_n}  \tcS_n\ar[r]^{\iota}&X\times_{\calN} X\\
}
\end{equation*}

 The action of the Hecke algebra $\calH_t$ on
$H_{\TT'}^*(\tcS_n)$ is induced by the pullback map
 $$\iota^{*}: H_{\TT'}^*(X \times_{\calN} X) \rightarrow H_{\TT'}^*(\tcS_n \times_{\calS_n} \tcS_n).$$  Since Theorem \ref{formula} identifies the action of the operator given by $\phi_* [\Mbar_{0,2}(X,\beta)]^{\mathrm{vir}}$,
 the corresponding result for $\tcS_n$ follows from the equality
$$\phi_{*} [\Mbar_{0,2}(\tcS_{n},\beta)]^{\mathrm{vir}}
= \phi_* \iota^{!} [\Mbar_{0,2}(X,\beta)]^{\mathrm{vir}}
 =\iota^{*}(\phi_{*} [\Mbar_{0,2}(X,\beta)]^{\mathrm{vir}}).$$
\end{proof}

\section{Appendix: second cohomology of Springer fibers}
In this appendix we would like to prove the following theorem:
\begin{thm}\label{simply-laced}
Let $G$ be a simply laced semi-simple group and let as before
$f:T^*\calB\to\calN$ be the Springer map. Then for any non-regular $n\in\calN$ the
restriction map $H^2(\calB,\ZZ)\to H^2(f^{-1}(n),\ZZ)$ is an isomorphism
(note that $f^{-1}(n)$ is embedded in $\calB$ by the definition of $f$).
\end{thm}
Theorem \ref{simply-laced} implies that the restriction map
$H^2(T^*\calB,\ZZ)\to H^2(\tcS_n,\ZZ)$ is an isomorphism
as was announced in the remark after Theorem \ref{formula-slodowy}.
\begin{proof}
Without loss of generality we may assume that $G$ is almost simple. We shall also
write $\calB_n$ instead of $f^{-1}(n)$.
Let $\calO_r$ denote the regular orbit in $\calN$ (i.e. the unique open orbit) and let
$\calO_{sr}\subset \calN$ denote the subregular nilpotent orbit (i.e. the unique orbit which has
codimension 2 in $\calN$). Let us first assume that $n\in\calO_{sr}$. Then it is well known that
$f^{-1}(n)$ is a tree of $\PP^1$'s of degrees corresponding to all
$\alp^{\vee}\in\Pi^{\vee}\subset\Lam^{\vee}$;
this implies that $H_2(\calB_n,\ZZ)=\Lam^{\vee}$ and $H^2(\calB_n,\ZZ)=\Lam$.
Since $H^2(\calB,\ZZ)=\Lam$ the statement of the Theorem follows.

For general $n\in \calN$ our strategy will be as follows. First we are going to show that
the natural map $H_2(\calB_n,\ZZ)\to H_2(\calB,\ZZ)$ is surjective. Then we going
to show that the natural map $H^2(\calB_n,\CC)\to H^2(\calB,\CC)$ is an isomorphism.
On the other hand, it is shown in \cite{DCLP} that $H^*(\calB_n,\ZZ)$ has no torsion, which
implies the statement of Theorem \ref{simply-laced}.

To prove the surjectivity of the map $H_2(\calB_n,\ZZ)\to H_2(\calB,\ZZ)$ we are going to show that
for every simple coroot $\alp^{\vee}\in \Pi^{\vee}$ there exists an embedding
$\PP^1\to \calB_n$, which has degree $\alp^{\vee}$ in $\calB$. The required surjectivity
follows since $H_2(\calB,\ZZ)=\Lam^{\vee}$ is generated by simple coroots. To show the existence of such
a $\PP^1$ let us denote by $\calP_{\alp}$ the moduli space of subminimal parabolics corresponding to $\alp$.
Then all the $\PP^1$'s of
degree $\alp^{\vee}$ in $\calB$ are fibers of the projection $p_{\alp}:\calB\to\calP_{\alp}$.
Since $n$ is not regular, it lies in the closure of $\calO_{sr}$. Thus there exists a curve $C$ with a point
$c\in C$ and a map $\eta:C\to\calN$ such that $\eta(x)\in \calO_{sr}$ if $x\neq c$ and $\eta(c)=n$.
For any $x\neq c$ the Springer fiber $\calB_{\eta(x)}$ contains unique $\PP^1$ of degree $\alp^{\vee}$.
In other words, we get a map $\phi:C\backslash\{ c\}\to \calP_{\alp}$ such $p_{\alp}^{-1}(\phi(x))\subset
\calB_{\eta(x)}$. Since $\calP_{\alp}$ is proper, the map $\phi$ extends to the whole of $C$ and we are going
to have $p_{\alp}^{-1}(\phi(c))\subset \calB_{\phi(c)}$. Since $\phi(c)=n$, we are done.

To prove that the map $H^2(\calB_n,\CC)\to H^2(\calB,\CC)$ is an isomorphism,
let us consider the Springer sheaf $\Spr=f_*(\CC[\dim \calN])$ (here $f_*$ denotes the derived
direct image in the category of constructible sheaves). It is well-known (cf. \cite{chriss-ginz}) that
it is perverse and semi-simple. Moreover, its irreducible direct summands are all of the form
$\IC(\calE)$ where $\calE$ is a $G$-equivariant local system on a $G$-orbit $\calO\subset \calN$
and $\IC(\calE)$ stands for its Goresky-Macpherson extension to $\calN$. We shall write
$\CC_{\calO}$ for the constant local system on $\calO$ with fiber $\CC$. Also, for
any complex of sheaves $\calF$ on $\calN$ we shall denote by $\calF_n$ the fiber of
$\calF$ at $n$; this is a complex of vector spaces over $\CC$.

Let now $n\in \calN$ which is not regular and not subregular
and let $\calO$ be its $G$-orbit. Then $H^2(\calB_n,\CC)$ is the cohomology of the
fiber of $\Spr$ at $n$ of degree $2-\dim\calN$.
Let $\calO'\subset\calN$ be a $G$-orbit such that $\bar{\calO}_{sr}\supset\bar{\calO'}\supset \calO$ and let
$\calE$ be a local system on $\calO'$. Then we claim that
$H^{2-\dim\calN}(\IC(\calE)_n)$ is $0$ unless $\calO'=\calO_{sr}$. Indeed, this follows from the definition of
IC-sheaves and from the fact that if $\calO'\neq \calO_{sr}$ then $\dim\calO'<\dim\calN-2$.
Thus in order the only contribution to $H^{2-\dim\calN}(\IC(\calE)_n)$ may come from
$\IC(\calE)$ where $\calE$ is a local system either on $\calO_r$ or $\calO_{sr}$.
It is well-known that the only $\calE$ on $\calO_r$ such that $\IC(\calE)$ appears in $\Spr$ is the
constant local system and in this case $\IC(\calE)$ is constant on $\calN$ and therefore does not
contribute anything to $H^{2-\dim\calN}(\IC(\calE)_n)$. Also, if $\calE$ is an irreducible local system on
$\calO_{sr}$ such that $\IC(\calE)$ appears in $\Spr$, then $\calE$ is constant
 and
$\Hom(\IC(\CC_{\calO_{sr}}[\dim\calO_{sr}]),\Spr)=\grt^*$
\footnote{Here we again use the fact that $G$ is simply laced}. On the other hand,
$H^{2-\dim\calN}(\IC(\CC_{\calO_{sr}}))=\CC_{\bar{\calO}_{sr}}$.
\footnote{Here $H^{2-\dim\calN}$ stands for the corresponding cohomology sheaf with respect to the usual
(not perverse) t-structure}
Thus $H^{2-\dim\calN}(\IC(\CC_{\calO_{sr}})_n)=\CC$ and hence $H^{2-\dim\calN}(\Spr_n)=\grt^*=H^2(\calB,\CC)$.

\end{proof}

\end{document}